\documentclass[12pt,a4paper]{article}

\usepackage{amssymb}
\usepackage{amsmath, amsthm}
\usepackage{arydshln}
\usepackage{epsfig}
\usepackage{setspace}

\usepackage{geometry}

\def\RR{{\mathbb{R}}}

\numberwithin{equation}{section}

\newcommand{\grad}{\mathop{\rm grad} }

\newtheorem{Thm}{Theorem}[section]
\newtheorem{Prop}[Thm]{Proposition}
\newtheorem{Lem}[Thm]{Lemma}
\newtheorem{Cor}[Thm]{Corollary}
\theoremstyle{definition}

\newtheorem{Rem}[Thm]{Remark}
\newtheorem{Ex}[Thm]{Example}

\title{On a manifold formulation\\ 
	of self-concordant functions}
\author{Hiroshi HIRAI \\
	Department of Mathematical Informatics, \\
	Graduate School of Information Science and Technology,   \\
	The University of Tokyo, Tokyo, 113-8656, Japan.\\
	\texttt{\normalsize hirai@mist.i.u-tokyo.ac.jp}
}
\begin{document}
	\maketitle
\begin{abstract} 
	In this paper, we address an extension of the theory of self-concordant functions 
	for a manifold. We formulate the self-concordance of a geodesically convex function
	by a condition of the covariant derivative of its Hessian, and 
	verify that many of the analogous properties, such as the quadratic convergence of Newton's method 
	and the polynomial iteration complexity of the path-following method, 
	are naturally extended. 
	However it is not known whether a useful class of self-concordant functions/barriers 
	really exists for non-Euclidean manifolds.
	To this question, 
	we provide a preliminary result that the squared distance function 
	in the hyperbolic space of curvature $- \kappa$ is $\sqrt{\kappa}/2$-self-concordant 
	and the associated logarithmic barrier of a ball of radius $R$ is 
	an $O(\kappa R^2)$-self-concordant barrier. We also give an application to the minimum enclosing ball in a hyperbolic space.
\end{abstract}
Keywords: geodesically convex optimziation, interior point method, 
self-concordant functions, information geometry
\section{Introduction}
The present paper addresses the power of {\em geodesically convex optimization}
from computational complexity perspectives.
This is motivated by the recent development 
on operator scaling and its generalization~\cite{AGLOW,BFGOWW,GGOW}, where 
several important problems in mathematical science
are shown to be formulated as geodesically convex optimization on Riemannian manifolds.
By first- or second-order methods on the manifolds, 
some of them are successfully solved in polynomial time~\cite{AGLOW,BFGOWW}.
On the other hand, due to the limitation of current methods~\cite{FR2021}, 
many of them still remain to desire a theoretically good algorithm.
In the literature, 
the design of optimization methods on manifolds, 
which utilizes geodesic convexity and has polynomial iteration complexity, 
is an important issue. Particularly, several papers 
asks an interior-point polynomial algorithm 
for geodesically convex optimization on manifolds; 
see \cite[p. 99]{BFGOWW} and see also \cite{BLNW,FR2021}.

For Euclidean convex optimization, interior-point methods 
provide a powerful framework to derive polynomial time complexity. 
The underlying mathematical basis is established 
by the theory of {\em self-concordant functions} 
due to Nesterov and Nemirovskii~\cite{NN94}.
The aim of the present paper is 
to examine an extension of self-concordant functions and 
the associated path-following method for a manifold setting.
Actually, 
such an attempt has already been addressed by Udriste~\cite{Udriste97} 
and Jiang, Moore, and Ji~\cite{JMJ07}. 
They considered self-concordant functions on general manifolds, 
and claimed several results analogous to the ones for the Euclidean case.
However, we found that the both papers 
contain several vague points on treatments of tensor fields and 
covariant differentiation $\nabla$.
For example, the both papers assume that trilinear form 
$(X,Y,Z) \mapsto \nabla_{X}\nabla_{Y} \nabla_{Z} f$ is symmetric.
This is not true, 
even if $\nabla$ is a ``symmetric" connection ($\nabla$ is torsion-free); 
see Lemma~\ref{lem:asymetric}.
Although their arguments 
may be justified by appropriate modifications and corrections, 
we here examine a ``faithful adaptation" 
of arguments in Nesterov's books~\cite{Nesterov04,Nesterov18} for a manifold setting, 
by means of the standard language of differential geometry.

The first contribution of this paper is to show 
that such an adaptation is actually possible.
A naive definition of the self-concordance (SC) for a Riemannian manifold $M$ is 
by differentiation via geodesic $\gamma:[0,1] \to M$ as
$|(f \circ \gamma)'''(t)| \leq 2 \sigma (f \circ \gamma)''(t)$.
We point out that, due to the asymmetry of $\nabla_{X}\nabla_{Y} \nabla_{Z} f$,  
a lemma of symmetric trilinear forms by Nesterov and Nemirovski~\cite[Proposition 9.1.1]{NN94} is not applicable.
Consequently, 
this naive definition is not enough to deriving desirable properties of SC.
We here employ a stronger condition as a definition of SC, 
which concerns the covariant derivative of the Hessian $Hf$ of $f$
and equivalents the original if $M = \RR^n$.
Then we verify several required properties of SC analogous to the Euclidean case.
Although many of presented arguments may be straightforward adaptations 
of those in \cite{Nesterov04,Nesterov18}, 
they clarify differential geometric nature of SC properties, which are  
hidden for the Euclidean case.
Particularly, we incorporate viewpoints of {\em information geometry}~\cite{AN00} 
to the quadratic convergence proof of Newton's method. 
We believe that these are insightful even in the original Euclidean setting. 

The second contribution of this paper addresses the next question:
For non-Euclidean manifolds,  
do nontrivial, or even useful, classes of SC functions really exist ? 
We show that
the squared distance function $x \mapsto \frac{1}{2} d(p,x)^2$
in the hyperbolic space of curvature $- \kappa < 0$
is $\sqrt{\kappa}/2$-self-concordant.
Together with the logarithmic barrier construction, 
any ball $B$ of radius $R$ admits an $O(\kappa R)$-self concordant barrier.
We exhibit that our results can be applied 
to the {\em minimum enclosing ball problem} in a hyperbolic space~\cite{NH15}.
This suggests the existence of an SC-barrier for a more general convex set 
and its further applications. 
On the other hand, since the order of the barrier parameter is tight, 
the presented barrier of a ball may be unsatisfactory;  
the iteration complexity of the path-following method involving this barrier
depends on radius $R$. 
This suggests a further difficulty of constructing a good SC-barrier, 
which is not easy already in the Euclidean case.
Nevertheless, we hope that these results will be a springboard for discussion
and stimulate further research in this direction.

\paragraph{Related work.}
Optimization on manifolds (Riemannian optimization) is one of active areas 
in continuous optimization.
In its modeling flexibility, various problems in machine learning, statistics, 
and control theory can be formulated in Riemannian optimization, 
and (practically) efficient algorithms have been developed; 
see textbooks \cite{AMS07, Boumal23,Sato21}.
The design of theoretically efficient optimization algorithms, 
with basis on geodesic convexity, 
is a challenging new direction of research; see \cite[Chapter 11]{Boumal23} 
and references therein.

The {\em ellipsoid method} is another powerful algorithm to 
derive polynomial-time complexity.
Rusciano~\cite{Rusciano19} introduced a cutting plane method, 
similar to the ellipsoid method, for Hadamard manifolds. 
He showed the existence of a cutting plane sequence 
such that the volume of the domain including an optimum decreases in linear order.
However, the sequence is non-constructive.
Also, as pointed out in \cite{FR2021},  
the volume of a ball in a hyperbolic space
depends exponentially on its diameter (see e.g., \cite[p. 189]{Sakai96}), 
and hence the iteration complexity depends on the diameter of the domain.  
So this algorithm, even if it could be implemented, 
has difficulty for problems on a large domain.

Already in Euclidean case, 
interior-point methods are closely linked to Riemannian geometry.
Nesterov and Todd~\cite{NT02} regarded the interior of the feasible region of LP 
as a Riemannian manifold with Hessian metric induced 
by a self-concordant barrier.
They connected the Riemannian length of the central curve
with the complexity of interior-point methods. 
Along this line, 
Ohara and Tsuchiya~\cite{OT07} developed 
an information geometric framework for LP, and provided 
a sharp iteration bound of an interior point method
by means of a differential geometric measure of the central curve; 
see \cite{KOT13} for further generalization.
We are partly inspired by their framework: 
Our setting is on a manifold $M$ with a connection $\nabla$ and 
a Riemannian metric  $\langle, \rangle_f$ given by the Hessian of an SC-function $f$.
Here $\nabla$ is not  necessarily the Levi-Civita connection of $\langle, \rangle_f$, 
and the {\em dual connection} $\nabla^*$,  
an information geometric concept, plays a role.

One of important aspects of self-concordant functions is 
its relation to the {\em Fenchel-Legendre duality} 
in convex analysis; see \cite[Section 5.1.5]{Nesterov18}.
Our presented framework lacks such an aspect.
Recently, several analogues of Fenchel-Legendre duality 
are considered for Hadamard manifolds~\cite{HH22,SBH22}.
It is an interesting direction to incorporate 
them into the framework or to take inspiration from them 
to introduce a more reasonable definition of self-concordance on manifolds.

\paragraph{Organization.} 
The rest of this paper is organized as follows. 
In Section~\ref{sec:pre}, we give preliminaries 
on linear algebra and differential geometry.
In Section~\ref{sec:SC}, we present the results on self-concordance.
In Section~\ref{sec:distance}, 
we present the results on squared distance functions on hyperbolic spaces.

\section{Preliminaries}\label{sec:pre}	
\subsection{Bilinear forms}
In this paper, all vector spaces are finite dimensional and over $\RR$.
For a vector space $V$, let $V^*$ 
denote the dual space of $V$, i.e., the vector space of all linear functions $V \to \RR$.
A vector $p$ in $V^*$ is also written as $p(\cdot)$; this means that 
the dot may be replaced by any vector $v \in V$, such as $p(v)$.
 
Let $U,V$ be vector spaces.
For a linear map $\tau:U \to V$, 
let $\tau^*:V^* \to U^*$ denote the dual map  defined by $p(\cdot) \mapsto p(\tau(\cdot))$.
For a multilinear form $G \in  V^* \otimes \cdots \otimes V^*$ 
on $V$, let $\tau^*G \in U^* \otimes \cdots \otimes U^*$ 
denote the multilinear form on $U$ defined by $\tau^* G := G(\tau(\cdot),\cdots,\tau(\cdot))$.  

A bilinear form $G \in V^* \otimes V^*$ is said to be {\em symmetric} if 
$G(u,v) = G(v,u)$ for $u,v \in V$.
A symmetric bilinear form $G \in V^* \otimes V^*$ is viewed as $G: V \to V^*$ 
by $v \mapsto G(v,\cdot) = G(\cdot, v)$. 
Suppose that $G$ is nondegenerate. 
The inverse map $G^{-1}: V^* \to V$ is  
viewed as a symmetric bilinear form on $V^*$ by 
$G^{-1}(p,q) := G(G^{-1}p, G^{-1}q)$.
If
$\tau:U \to V$ is a linear isomorphism, 
then it holds, 
\begin{equation}\label{eqn:tau}
	(\tau^* G)^{-1} = \tau^{-1} G^{-1},
\end{equation}
where $(\tau^{-1})^{**} = \tau^{-1}$. 
Indeed, it holds $v \mapsto (\tau^*G) v = G (\tau(v), \tau(\cdot)) = \tau^* (G \tau v)$, 
and hence $p \mapsto (\tau^* G)^{-1}p = \tau^{-1} (G^{-1} (\tau^*)^{-1} p) 
= G^{-1} ((\tau^{-1})^*(p), (\tau^{-1})^*(\cdot))$.

%
 
For two symmetric bilinear forms $F,G$ on $U$, 
the notation $F \preceq G$ means that $F(u,u) \leq G(u,u)$ for all $u \in U$.
If $F\succeq 0$, then $F$ is called positive semidefinite.
In addition, if $F$ is nondegenerate, $F$ is called positive definite.
Let $G$ be a positive definite bilinear form on $U$. 
Then $G$ defines the inner product 
$\langle , \rangle$ on $U$ by $\langle u,v \rangle :=G(u,v)$, and the norm $\|u\| := \sqrt{\langle u,u\rangle}$ accordingly.
Let $F: U \to U^*$ be another symmetric bilinear form. 
The linear transformation $G^{-1}F: U \to U$ is {\em symmetric} with respect to $\langle, \rangle$, 
i.e., $\langle u, G^{-1} F v \rangle = \langle G^{-1}F u, v \rangle (= F(u,v))$.
In particular, $G^{-1}F$ can be diagonalized with real eigenvalues and orthonormal eigenvectors.
Then 
\begin{equation}\label{eqn:G^-1F}
	F \preceq G \Leftrightarrow \mbox{each eigenvalue of $G^{-1}F$ is at most $1$}.	
\end{equation}
Suppose, in addition, that $F$ is also positive definite.
Then it holds
\begin{equation}\label{eqn:F<G}
F \preceq G \Leftrightarrow F^{-1} \succeq G^{-1}.	
\end{equation}
These can be seen from 
the matrix representation $\tilde F =(F(e_i,e_j)), \tilde G = (G(e_i,e_j))$ of $F,G$ with respect to an orthonormal basis for $\langle,\rangle$ (so that $\tilde G = \tilde G^{-1} = I$).

Additionally, $G$ defines an inner product $\langle , \rangle^*$ and norm $\|\cdot \|^*$ of the dual space $U^*$
by $\langle p,q\rangle^* := G^{-1}(p,q)$ and $\|p\|^* := \sqrt{\langle p,p\rangle^*}$.
Then the following version  of Cauchy-Schwarz inequality holds:
\begin{equation}\label{eqn:CM}
	p(u) \leq \|p\|^*\|u\| \quad (u \in U, p \in U^*).
\end{equation}
This follows from $p(u) = (GG^{-1}p)(u) = G(G^{-1}p,u) 
= \langle G^{-1}p, u\rangle \leq \|G^{-1}p\|\|u\|$ 
and $\|G^{-1}p\|^2 = G(G^{-1}p,G^{-1}p) = (\|p\|^*)^2$.

\subsection{Manifold with an affine connection}
We here introduce basic terminologies for smooth manifolds; 
our references are \cite{Boumal23,Sakai96}.
Throughout the paper,
sufficient smoothness is assumed for functions, maps, and vector/tensor fields on manifolds. 
Let $M$ be an $n$-dimensional manifold.
For $x \in M$, let $T_x = T_x(M)$ denote the tangent space at $x$.
For a vector field $X$ on $M$, let $X_x \in T_x$ denote the vector at $x$.
For a map $\varphi: M \to N$ and $x \in M$, where $N$ is a manifold, 
let $D\varphi(x): T_x(M) \to T_{\varphi(x)}(N)$ denote the differential of $\varphi$ at $x$.
For a curve $\gamma:[a,b] \to M$, 
let $\dot \gamma(t) := D\gamma(t) (d/dt)$ 
be the vector tangent to $\gamma$ at $t \in [a,b]$.

Suppose that $M$ has an affine connection $\nabla$ on tangent bundle $T(M) = \bigcup_{x \in M} T_x$.
Given a curve $\gamma:[a,b] \to M$ with $x = \gamma (a)$, 
the connection $\nabla$ determines 
the {\em parallel transport} along $\gamma$, 
which is a linear isomorphism $\tau_{\gamma,t}$ 
from $T_x$ to $T_{\gamma(t)}$ for $t \in [a,b]$.
For a vector field $Y$ and vector $u \in T_x$ at $x \in M$, 
the {\em covariant derivative} 
$\nabla_u Y$ of $Y$ by $u$  
is given by $\nabla_u Y = \lim_{t \to 0} (\tau_{\gamma,t}^{-1} Y_{\gamma(t)} - Y_{x})/t$, 
where $\gamma:[0,l] \to M$ is any smooth curve with $\gamma(0) = x$ 
and $\dot \gamma(0) = u$. 
The covariant derivative $\nabla_X Y$ of $Y$ by a vector field $X$
 is a vector field so that $(\nabla_X Y)_x = \nabla_{X_x} Y$. 
 The {\em geodesic} is a curve $\gamma:[0,l] \to M$ 
 such that $\dot \gamma(t) = \tau_{\gamma,t}(\dot \gamma(0))$.
 A geodesic $\gamma$ is a solution of differential equation 
 $\nabla_{\dot\gamma(t)}\gamma(t) =0$, and 
 it is uniquely determined from initial point $x = \gamma(0)$ and direction $u=\dot\gamma(0)$.
 Particularly, given $x \in M$ and $u \in T_x$, 
 there is a geodesic $\gamma:[0,l] \to M$ with $x = \gamma(0)$ and direction $u=\dot\gamma(0)$. 
 In this case, we write $\gamma(t) = \exp_x t u$.
 The map $u \mapsto \exp_x 1 u$ is defined on 
 a neighborhood of the origin of $T_x$, and 
 is called the {\em exponential map}.

The covariant derivative is extended for tensor fields;
see \cite[Section 10.7]{Boumal23} and \cite[II.1.3]{Sakai96}.
In this paper, we consider it for covariant tensor fields.
For a $(0,k)$-tensor field $G=(G_x)_{x \in M}$, 
where $G_x \in T_x^{*} \otimes \cdots \otimes T_x^*$, 
the covariant derivative $\nabla_u G$ by $u \in T_x$ is defined by
\begin{equation}\label{eqn:nabla_uG}
\nabla_u G  :=  \left.\frac{\rm d}{{\rm d}t}\tau_{\gamma,t}^* G_x\right|_{t=0},
\end{equation}
where $\gamma:[0,l] \to M$ is any curve with $\gamma(0) = x$ and $\dot{\gamma}(0) = u$.
The covariant derivative $\nabla_X G$ by vector field $X$ is defined as 
a $(0,k)$-tensor field obtained by $\nabla_X G = (\nabla_{X_x}G)_{x \in M}$.
The covariant derivative $\nabla_X G$ can also be written as (or defined by)
\begin{equation}
	(\nabla_X G)(Y_1,\ldots,Y_k) = X(G(Y_1,\ldots,Y_k)) - \sum_{i=1}^k G(Y_1,\ldots, \nabla_X Y_i,\ldots,Y_k),
\end{equation}
where $X,Y_1,\ldots,Y_k$ are vector fields.
Then
$(Y_1,\ldots,Y_k,X) \mapsto (\nabla_X G)(Y_1,\ldots,Y_k)$ is a $(0,k+1)$-tensor field.

In this paper, we assume that $(M,\nabla)$ satisfies the following conditions: 
\begin{itemize}
	\item $\nabla$ is torsion-free:
	$\nabla_X Y - \nabla_Y X = [X,Y]$ holds for vector fields $X,Y$.
	\item For each point $x \in M$, 
	the exponential map $\exp_x$ is defined on $T_x$ 
	and is a surjection from $T_x$ to $M$.
\end{itemize}
We basically consider the case where $M$ is a complete Riemannian manifold 
and $\nabla$ is the Levi-Civita connection; 
then $(M,\nabla)$ satisfies the above conditions; see \cite[III.1]{Sakai96}.
However, by its affine invariant nature, 
our argument in Section~\ref{sec:SC}
does not use the intrinsic Riemannian metric $\langle, \rangle$ of $M$, 
and also consider another metric $\langle,\rangle_f$ 
induced by a smooth function $f$ in question.
So we do not make $\langle, \rangle$ explicit.

Let $f: M \to \RR$ be a smooth function. 
The differential $Df = (Df_x)_{x \in M}$ of $f$ is a $(0,1)$-tensor field, where 
$Df_x =  Df(x): T_x \to T_{f(x)}(\RR) = \RR$ is given by 
\begin{equation}
Df_x(u) = \left.\frac{\rm d}{{\rm d}t} f(\gamma(t))\right|_{t=0} \quad (u \in T_x)
\end{equation}
 for any curve $\gamma:[0,l] \to \RR$ with $\dot \gamma(0) = u$.
 If $f$ is viewed as a $(0,0)$-tensor, then
$Df(X)$ equals $\nabla_X f$. 
A further differentiation of $Df$ yields  
a $(0,2)$-tensor field $Hf = (Hf_x)_{x \in M}$, called the {\em Hessian} of $f$, as 
\begin{equation}
	Hf(X,Y) := (\nabla_X Df) (Y) = X(Df(Y)) - Df(\nabla_X Y), 
\end{equation}
where $X,Y$ are vector fields.
In the case where $\nabla$ is the Levi-Civita connection 
of a Riemannian manifold, this matches the usual definition
of the Hessian
$(X,Y) \mapsto \langle \nabla_X \grad f, Y \rangle$; see \cite[Example 10.78]{Boumal23}.

The bilinear form $Hf_x \in T^{*}_x \otimes T^*_x$ 
is called the Hessian of $f$ at $x \in M$.
The following is well-known. See \cite[Proposition 5.15]{Boumal23}.
\begin{Lem}
	$Hf_x$ is a symmetric bilinear form on $T_x$.
\end{Lem}
\begin{proof}
	This follows from $(\nabla_X Df)(Y) = X (Df(Y)) - Df(\nabla_X Y) =  X Y f - Df(\nabla_Y X + [X,Y]) = (XY - [X,Y]) f - Df(\nabla_Y X)= YXf - Df(\nabla_Y X) = (\nabla_Y Df)(X)$.
\end{proof}
The diagonal of the Hessian is given as follows. See \cite[Eaxmple 10.81]{Boumal23}.
\begin{Lem}\label{lem:Hf(u,u)} For $x \in M$ and $u \in T_x$, we have
	\[
	Hf_x(u,u) = \left.\frac{{\rm d}^2}{{\rm d}t^2} f(\exp_x tu) \right|_{t=0}.
	\]
\end{Lem}	
\begin{proof}
	 For $\gamma(t) = \exp_x tu$ we have $\dot \gamma(t)= \tau_{\gamma,t} u$ and
	 $Df_{\gamma(t)} (\tau_{\gamma,t} u) = Df_{\gamma(t)} (\dot \gamma (t)) = \frac{\rm d}{{\rm d}t}f(\gamma(t))$. Apply (\ref{eqn:nabla_uG}).
\end{proof}
We consider the covariant derivative of $(0,2)$-tensor field $Hf$, which provides
a $(0,3)$-tensor field $(X,Y,Z) \mapsto (\nabla_X Hf)(Y,Z)$.
Accordingly, we have a trilinear form $(u,v,w) \mapsto (\nabla_u H)(v,w)$ on each $T_x$.
In contrast to the Euclidean case, 
this tri-linear form is not necessarily symmetric, 
due to nonzero curvature; see e.g., \cite[Exercise 10.87]{Boumal23}. 
	Recall the curvature, which is a $(1,3)$-tensor field $(X,Y,Z) \mapsto R(X,Y)Z$ defined by
	\begin{equation}
	R(X,Y)Z  := \nabla_X \nabla_Y Z - \nabla_Y \nabla_X Z - \nabla_{[X,Y]} Z. 
	\end{equation}
\begin{Lem}\label{lem:asymetric}
	$
	(\nabla_X Hf) (Y,Z) = (\nabla_Y Hf)(X,Z) - (R(X,Y)Z) f.
	$ 
\end{Lem}	
\begin{proof}
	By definition, we have
	\begin{eqnarray*}
		&& (\nabla_X Hf)(Y,Z) =  X (Hf(Y,Z)) - Hf(\nabla_X Y,Z) - Hf(Y, \nabla_X Z) \\
		&& =  X (Y Zf - Df (\nabla_Y Z)) - (\nabla_X Y) Df(Z) + Df(\nabla_{\nabla_X Y}Z) \\
		&& \quad - Y Df(\nabla_X Z) + Df(\nabla_Y \nabla_X Z). 
	\end{eqnarray*}
By changing role of $X,Y$, we have the expression of $(\nabla_Y Hf)(X,Z)$, and obtain
	\begin{eqnarray*}
	&& (\nabla_X Hf)(Y,Z) - (\nabla_Y Hf)(X,Z) 	\\
	&&= ([X,Y] - \nabla_X Y + \nabla_Y 
	X)Zf  + \nabla_{\nabla_X Y - \nabla_Y X}Z f  + \nabla_Y \nabla_X Z f -  	\nabla_X \nabla_Y Z f \\
	&& = - R(X,Y)Z f. 
\end{eqnarray*}	
\end{proof}
The diagonal of $\nabla_X Hf(Y,Z)$ is
given by differentiation along geodesic; see \cite[Example 10.81]{Boumal23}.
\begin{Lem}\label{lem:nabla_uHf(u,u)}  For $x \in M$ and $u\in T_x$, we have
\[
(\nabla_u Hf)(u,u) = \left.\frac{{\rm d}^3}{{\rm d}t^3} f(\exp_x tu) \right|_{t=0}.
\]
\end{Lem}
\begin{proof}
	For $\gamma = \exp_x tu$, we have
	$\frac{\rm d}{{\rm d}t}Hf_{\gamma(t)}(\tau_{\gamma,t} u, \tau_{\gamma,t} u) =  
	\frac{\rm d}{{\rm d}t} Hf_{\gamma(t)}(\dot \gamma(t), \dot \gamma(t)) = \frac{\rm d^3}{{\rm d}t^3}f(\gamma(t))$.
\end{proof}

\section{Self-concordant functions}\label{sec:SC}
Let $M$ be a manifold with an affine connection $\nabla$ (satisfying the above conditions).
A subset $C \subseteq M$ is said to be {\em convex} (or more precisely {\em geodesically convex})
if for every geodesic $\gamma:[a,b] \to M$ with $\gamma(a), \gamma(b) \in C$
it holds $\gamma(t) \in C$ for every $t \in [a,b]$.
Accordingly, a function $f: M \to \RR \cup \{+\infty\}$
is said to {\em convex} if for every geodesic $\gamma:[a,b] \to M$
the function $f\circ \gamma:[a,b] \to \RR \cup \{+\infty\}$ is convex. 

Let $C \subseteq M$ be an open convex subset, 
which is viewed as a submanifold of $(M,\nabla)$.
\begin{Lem}
	$f:C \to \RR$ is convex if and only if the Hessian $Hf_x$
	is positive semidefinite for all $x \in C$.
\end{Lem}
\begin{proof}
	This follows from Lemma~\ref{lem:Hf(u,u)} and:
	$f$ is convex  $\Leftrightarrow$ $(f \circ \gamma)''(t) \geq 0$ 
	for every geodesic $\gamma$.
\end{proof}
We say that a convex function $f:C \to \RR$ is {\em nondegenerate} if $Hf_x$ is positive definite
for all $x \in C$, and {\em closed} (or lower semicontinuous) if its epigraph $\{(x,y) \in C \times \RR \mid y \geq f(x)\}$
is a closed set.
We often omit the subscript $x$ of tensor $G_x$ at $x \in M$ if $x$ is clear.
\subsection{Definitions and properties}
Let $C \subseteq M$ be an open convex set. 
A closed and nondegenerate convex function 
$f:C \to \RR$ is called {\em $\sigma$-self-concordant ($\sigma$-SC)} for $\sigma \geq 0$
if it satisfies
\begin{equation}\label{eqn:SC}
|(\nabla_v Hf)(u,u)| \leq 2 \sigma Hf(v,v)^{1/2} Hf(u,u) \quad (u,v \in T_x, x\in C),
\end{equation}
and is called {\em weakly $\sigma$-self-concordant ($\sigma$-WSC)}
if it satisfies
\begin{equation}\label{eqn:WSC}
|(\nabla_u Hf)(u,u)| \leq 2 \sigma Hf(u,u)^{3/2} \quad (u \in T_x, x \in C).
\end{equation}
By Lemmas~\ref{lem:Hf(u,u)} and \ref{lem:nabla_uHf(u,u)}, 
the latter condition (\ref{eqn:WSC}) is equivalent to:
\begin{equation}\label{eqn:WSC_d}
\left| (f \circ \gamma)''' (t) \right| \leq 2 \sigma 
(f \circ \gamma)'' (t)^{3/2} 
\quad (\gamma: \mbox{geodesic in $C$}).
\end{equation}
The condition (\ref{eqn:WSC_d}) 
is a natural  analogue of the original definition 
of the self-concordance by Nesterov and Nemirovskii~\cite{NN94}.
So papers~\cite{JMM08, JMJ07, Udriste97} consider (\ref{eqn:WSC}) 
as the definition of the self-concordance on manifolds; see also the footnote of \cite[p. 31]{BFGOWW}.

In Euclidean space $M = \RR^n$, 
the self-concordance and weak self-concordance are the same. 
Indeed, $(X,Y,Z) \mapsto \nabla_X Hf(Y,Z)$ is symmetric (by $R= 0$ and Lemma~\ref{lem:asymetric}).
Then, by Nesterov-Nemirovski lemma~\cite[Appendix 1]{NN94} 
for symmetric multilinear forms, 
(\ref{eqn:WSC}) implies (\ref{eqn:SC}).
However,
for a manifold with nontrivial curvature $R$,
the trilinear form $\nabla_X Hf(Y,Z)$ is not necessarily symmetric, 
and the Nesterov-Nemirovski lemma is not applicable.
We see in Section~\ref{sec:distance} that 
SC and WSC are actually different.
Note that $\nabla_X Hf(Y,Z)$ is symmetric on $Y,Z$, 
and hence by the lemma the condition (\ref{eqn:SC}) 
is equivalent to $|\nabla_{u} Hf(v,w)| \leq 2\sigma Hf(u,u)^{1/2} Hf(v,v)^{1/2} Hf(w,w)^{1/2}$.

The behavior of a weighted sum of (W)SC-functions is the same as 
in the Euclidean case \cite[Theorem 4.1.1]{Nesterov04}(\cite[Theorem 5.1.1]{Nesterov18}), 
as follows. 
\begin{Lem}\label{lem:sum}
	For $i=1,2$, let $f_i: C_i \to \RR$ be $\sigma_i$-(W)SC. Suppose that $C_1 \cap C_2$ has nonempty interior.
	For $\alpha, \beta > 0$, the function $\alpha f_1 + \beta f_2: C_1 \cap C_2 \to \RR$ is
	$\sigma$-(W)SC for $\sigma := \max \{ \sigma_1/\sqrt{\alpha}, \sigma_2/\sqrt{\beta} \}$
\end{Lem}
For the WSC-case, the proof is also exactly the same as in the Euclidean case, and is omitted. 
For the SC-case, it needs a little modification. 
\begin{proof}
	Let $f := \alpha f_1 + \beta f_2$.
	For $x \in C_1 \cap C_2$ and $u,v \in T_x$, we have
	\begin{eqnarray}
		\frac{|\nabla_v Hf(u,u)|}{2Hf(u,u)^{1/2}Hf(v,v)} &\leq& \frac{\alpha |\nabla_v Hf_1(u,u)|+ \beta |\nabla_v Hf_2(u,u)|}{2 \sqrt{\alpha Hf_1(v,v)+ \beta Hf_2(v,v)}(\alpha Hf_1(u,u)+ \beta Hf_2(u,u))} \nonumber \\
		& \leq& \frac{\sqrt{\alpha} \sigma_1 x_1 \omega_1 + \sqrt{\beta} \sigma_2  x_2 \omega_2}
		{\sqrt{x_1^2+ x_2^2}(\alpha \omega_1+ \beta \omega_2)}, \label{eqn:quantity}
	\end{eqnarray}
where we let  $x_i := \sqrt{\alpha} Hf_i(v,v)^{1/2}$ and $\omega_i := Hf_i(u,u)$ for $i=1,2$.
We bound the maximum of this quantity (\ref{eqn:quantity}).
Observing invariance under the change $(x_1,x_2, \omega_1,\omega_2) \to (sx_1,sx_2,t \omega_1,t\omega_2)$ 
for $s,t > 0$, we may consider 
the following optimization problem:
\begin{eqnarray*}
	\mbox{Max.} && \sqrt{\alpha} \sigma_1 \omega_1 x_1 + \sqrt{\beta} \sigma_2 \omega_2 x_2 \\
	\mbox{s.t.} && x_1^2 + x_2^2 = 1, \alpha \omega_1 + \beta \omega_2 = 1, \\ 
	&&  x_1,x_2, \omega_1,\omega_2 \geq 0.
\end{eqnarray*}
Under fixing $\omega_i$, maximize this by $x_i$. 
This is a linear optimization over unit circle, 
and hence the optimum is given by $(x_1,x_2) = (\sqrt{\alpha}\sigma_1 \omega_1, \sqrt{\beta}\sigma_2 \omega_2)/ \sqrt{\alpha \sigma_1^2 \omega_1^2 + \beta \sigma_2^2 \omega_2^2}$.
The problem reduces to
	\[
	\mbox{Max.}  \sqrt{\alpha \sigma_1^2 \omega_1^2 
			+ \beta \sigma^2_2 \omega^2_2} \ \ 
		\mbox{s.t.} \ \ \alpha \omega_1 + \beta \omega_2 = 1, \omega_1,\omega_2 \geq 0.
	\]
	The maximum is attained by $(1/\alpha,0)$ or $(0, 1/\beta)$, and
	(\ref{eqn:quantity}) $\leq \max (\sigma_1/\sqrt{\alpha}, \sigma_2/\sqrt{\beta})$, as required.
\end{proof}

Let $f: C \to \RR$ be WSC. 
The Hessian $Hf_x$ of $f$ gives an inner product $\langle, \rangle_{f,x}$ on each $T_x$:
\begin{equation}
	\langle u, v \rangle_{f,x} := Hf_x(u,v) \quad (u,v \in T_x).
\end{equation}
Accordingly, we define the {\em local norm} by
$\|u\|_{f,x} := \sqrt{\langle u,u\rangle_{f,x}}$.
The manifold $C$ is viewed as a Riemannian manifold with respect to 
this inner product $\langle, \rangle_{f}$, 
although $\nabla$ is not necessarily the Levi-Civita connection of 
this Riemannian structure. 

\begin{Rem}\label{rem:information_geometry}
	This setting may be discussed from viewpoints 
	of information geometry and Hessian geometry~\cite{AN00,Shima07}.
We can consider the {\em dual connection} $\nabla^*$ 
via $\langle Y,\nabla^*_X Z \rangle_f 
:= X \langle Y,Z\rangle_f - \langle Y,\nabla_X Z \rangle_f$.
The symmetry of $(0,3)$-tensor $\nabla_X Hf(Y,Z)$ corresponds 
to the Codazzi equation for metric $\langle,\rangle_f$, 
and to $\nabla^*$ being torsion-free.
If $R = 0$ ($M$ is flat), 
then $(M, \langle,\rangle_f, \nabla,\nabla^*)$ is a dually-flat statistical manifold---
a well-studied and central concept in information geometry. 
The main focus in this paper is the case of $R \neq 0$. 
In the view of Lemma~\ref{lem:asymetric}, 
$\nabla^*$ is hardly torsion-free.
Information geometry methods for such situations   
seem less developed.
\end{Rem}

For $x \in C$ and $r \geq 0$, the {\em Dikin ellipsoid} $W^0(x;r)$ 
is defined by
\[
W^0(x;r) := \{ \gamma(1) \mid \mbox{$\exists$ geodesic }\gamma:[0,1] \to C, \gamma(0)=x, \|\dot\gamma(0)\|_{f,x} < r \},
\]
which is the image of the open ball of radius $r$ in $T_x$ by $\exp_x$.
In this setting, 
some of basic properties and arguments of (W)SC in the Euclidean case 
can be extended in a straightforward manner.
The following three propositions and their proofs 
are direct adaptations of those in \cite{Nesterov04,Nesterov18};
they are basically replacements of type:
\begin{eqnarray*}
&&	\langle \nabla f(x),h \rangle  =  \left. \frac{\rm d}{{\rm d}t} f(x+th) \right|_{t = 0} 
	 \quad \Longrightarrow \quad   
	Df_{\gamma(0)}(\dot \gamma(0)) = \left. \frac{\rm d}{{\rm d}t} f(\gamma(t)) \right|_{t = 0}, \\
&&  \langle \nabla f(x+h) - \nabla f(x), h \rangle = \int_{0}^1 \langle \nabla^{2} f(x+   th)h,h \rangle {\rm d}t \\
&& \quad \Longrightarrow \quad Df_{\gamma(1)}(\dot \gamma(1)) - Df_{\gamma(0)}(\dot \gamma(0)) = \int_{0}^1 Hf_{\gamma(t)}(\dot\gamma(t),\dot\gamma(t)) {\rm d}t.
\end{eqnarray*}
where $\gamma:[0,1] \to M$ is a geodesic.
The first one corresponds to \cite[Theorem 4.1.5]{Nesterov04} (\cite[Theorem 5.1.5]{Nesterov18}). 

\begin{Prop}\label{prop:WSC1} Suppose that $f$ is $\sigma$-WSC.
	\begin{itemize}
	\item [(1)] For any $x \in C$, it holds
	$
	W^0(x; 1/\sigma) \subseteq C.
	$  
\item[(2)] For a geodesic $\gamma: [0,1] \to C$,   
	it holds
	\begin{equation}\label{eqn:WSC_gamma_1}
	\|\dot\gamma(1)\|_{f,\gamma(1)} \geq \frac{\|\dot\gamma(0)\|_{f,\gamma(0)}}{1+ \sigma \|\dot\gamma(0)\|_{f,\gamma(0)}}.
	\end{equation}

	In addition, if $\gamma(1) \in W^0(\gamma(0);1/\sigma)$, then 
	\begin{equation}\label{eqn:WSC_gamma_2}
	 \|\dot\gamma(1)\|_{f,\gamma(1)} \leq \frac{\|\dot\gamma(0)\|_{f,\gamma(0)}}{1- \sigma \|\dot\gamma(0)\|_{f,\gamma(0)}}.
	\end{equation}
	\end{itemize}
\end{Prop}
\begin{proof} (1).
	Consider a geodesic $\gamma:[0,1] \to M$ with $\gamma(0) = x$.
	Suppose that for $a \in (0,1]$ the point $\gamma(a)$ is in the boundary of $C$. 
	Define a univariate function $\phi:[0,a) \to \RR$ by 
	\begin{equation}
		\phi(t) := 1/ \|\dot \gamma(t)\|_{f,\gamma(t)} = Hf(\dot \gamma(t),\dot \gamma(t))^{-1/2}= (f\circ\gamma)''(t)^{-1/2}.
	\end{equation}
	Then we have
	\[
	\left|\phi'(t)\right| = \frac{\left| (f\circ\gamma)'''(t)\right|}{2 (f \circ \gamma)''(t)^{3/2}} \leq \sigma.
	\]
    By integrating $-\sigma \leq \phi'(t) \leq \sigma$ from $0$ to $a$, 
    we have 
    \begin{equation}\label{eqn:phi}
    	\phi(0) - \sigma a \leq \lim_{t \to a} \phi(t) \leq \phi(0) + \sigma a
    \end{equation}
	Since $f$ is lower semicontinuous and $C$ is open, it holds $f(\gamma(t)) \to \infty$ for $t \to a$. Then $\lim_{t\to a} \phi(t) = 0$. 
	Hence it must hold $a \geq \phi(0)/\sigma$.
   	If $\phi(0)/\sigma > 1 \Leftrightarrow \|\dot \gamma(0)\|_{f,x} < 1/\sigma$, 
   	then $\gamma(1)$ belongs to the domain $C$ of $f$. This implies (1).
	
	(2). From (\ref{eqn:phi}) and $\gamma(1) \in C$, we have $\phi(1) \leq \phi(0) + \sigma$, 
	which is equivalent to $1/\|\dot\gamma(1)\|_{f,x} \leq 1/\|\dot\gamma(0)\|_{f,x}+\sigma$.
	This is (\ref{eqn:WSC_gamma_1}).
	The second equation (\ref{eqn:WSC_gamma_2}) follows similarly from $\phi(1) \geq \phi(0) - \sigma$.
\end{proof}

The next corresponds to \cite[Theorems 4.1.7 and 4.1.8]{Nesterov04} (\cite[Theorems 5.1.8 and 5.1.9]{Nesterov18}).

\begin{Prop}\label{prop:WSC2} Suppose that $f$ is $\sigma$-WSC.
	For a geodesic $\gamma:[0,1] \to C$,  it holds 
	\begin{eqnarray}
		&& Df_{\gamma(1)}(\dot{\gamma}(1)) - Df_{\gamma(0)}(\dot{\gamma}(0)) 
		\geq \frac{\|\dot\gamma(0)\|^2_{f,\gamma(0)}}{1+ \sigma \|\dot\gamma(0)\|_{f,\gamma(0)}}, 
		\label{eqn:WSC2_Df=>}\\
		&& f(\gamma(1)) \geq f(\gamma(0)) + Df_{\gamma(0)}(\dot{\gamma}(0)) + \frac{1}{\sigma^2} \omega(\sigma \|\dot\gamma(0)\|_{f,\gamma(0)}),
	\end{eqnarray}
	where $\omega(t) := t - \log (1+t)$.
	In addition, if $\gamma(1) \in W^{\circ}(\gamma(0);1/\sigma)$, then it holds
		\begin{eqnarray}
		&& Df_{\gamma(1)}(\dot{\gamma}(1)) - Df_{\gamma(0)}(\dot{\gamma}(0)) 
		\leq \frac{\|\dot\gamma(0)\|^2_{f,\gamma(0)}}{1- \sigma \|\dot\gamma(0)\|_{f,\gamma(0)}}, \label{eqn:WSC2_Df}\\
		&& f(\gamma(1)) \leq f(\gamma(0)) + Df_{\gamma(0)}(\dot{\gamma}(0)) + \frac{1}{\sigma^2} \omega_{*}(\sigma \|\dot\gamma(0)\|_{f,\gamma(0)}), \label{eqn:WSC2_f}
	\end{eqnarray}
	where  $\omega_*(t) := - t - \log (1-t)$.
\end{Prop}
\begin{proof}
	We only show the latter part; the former is similar.
	\begin{eqnarray*}
		&& Df_{\gamma(1)}(\dot{\gamma}(1)) - Df_{\gamma(0)}(\dot{\gamma}(0)) = 
		\int_0^1 \frac{\rm d}{{\rm d}t} Df_{\gamma(t)}(\dot{\gamma}(t)) {\rm d}t = \int_0^1 Hf_{\gamma(t)}(\dot{\gamma}(t),\dot{\gamma}(t)) {\rm d}t \\
		&& = \int_0^1 \| \dot{\gamma}(t) \|^2_{f,\gamma(t)} {\rm d}t  \leq \int_0^1 \frac{\|\dot\gamma(0)\|^2_{f,\gamma(0)}}{(1 - \sigma t \|\dot\gamma(0)\|_{f,\gamma(0)})^2} {\rm d}t, 
	\end{eqnarray*}
where the inequality follows by applying (\ref{eqn:WSC_gamma_2}) for geodesic $s \mapsto \gamma(ts)$, and the last integration equals the RHS of (\ref{eqn:WSC2_Df}).
By using (\ref{eqn:WSC_gamma_2}) for geodesic $s \mapsto \gamma(ts)$, we have
\begin{eqnarray*}
	 && f(\gamma(1)) - f(\gamma(0)) - Df_{\gamma(0)}(\dot{\gamma}(0)) = 
	 \int_{0}^{1} \frac{\rm d}{{\rm d}t} f(\gamma(t)) {\rm d}t  - Df_{\gamma(0)}(\dot{\gamma}(0)) \\
	 && = \int_{0}^1 Df_{\gamma(t)}(\gamma(t)) - Df_{\gamma(0)}(\dot{\gamma}(0)) {\rm d}t
	 \leq \int_{0}^1 \frac{t \|\dot\gamma(0)\|^2_{f,\gamma(0)}}{1- \sigma t\|\dot\gamma(0)\|_{f,\gamma(0)}} {\rm d}t, 
\end{eqnarray*}
where the last integration equals the third term of RHS of (\ref{eqn:WSC2_f}).
\end{proof}

We next consider properties of SC that cannot be derived from WSC.
The following corresponds to 
\cite[Theorem 4.1.6 and Corollary 4.1.4]{Nesterov04} (\cite[Theorem 5.1.7 an Corollary 5.1.5]{Nesterov18}).
\begin{Prop}\label{prop:SC}
	Suppose that $f$ is $\sigma$-SC.
	For a geodesic $\gamma:[0,1] \to C$ with $\gamma(1) \in W^{0}(\gamma(0);1/\sigma)$, it hold
	\begin{equation}\label{eqn:SC_Hf}
	(1- \sigma \|\dot \gamma(0)\|_{f,\gamma(0)})^2 Hf_{\gamma(0)} \preceq \tau_{\gamma,1}^* Hf_{\gamma(1)}	\preceq 
	\frac{1}{(1- \sigma \|\dot\gamma(0)\|_{f,\gamma(0)})^2}
	 Hf_{\gamma(0)},
	\end{equation}
	and 
	\begin{eqnarray} 
		&& \left\{ 
		1- \sigma \|\dot\gamma(0)\|_{f,\gamma(0)} + \frac{1}{3}\sigma^2 \|\dot\gamma(0)\|_{f,\gamma(0)}^2 \right\} Hf_{\gamma(0)}  \nonumber \\
		&& \quad \quad \preceq
		\int_{0}^1 \tau_{\gamma,t}^* Hf_{\gamma(t)} {\rm d}t  
		\preceq \frac{1}{1- \sigma \|\dot\gamma(0)\|_{f,\gamma(0)}} Hf_{\gamma(0)}.  \label{eqn:SC_Hf_integral}
	\end{eqnarray}
\end{Prop}
\begin{proof}
	Let $u \in T_x$.
	Define $\psi: [0,1] \to \RR$ by
	\begin{equation*}
		\psi(t) := \| \tau_{\gamma,t} u\|^2_{f,\gamma(t)} = \tau^{*}_{\gamma,t} Hf_{\gamma(t)}(u,u).
	\end{equation*}
Then we have
\begin{eqnarray*}
	\left|\psi'(t)\right| &=& \left|(\nabla_{\dot\gamma(t)} Hf_{\gamma(t)})(\tau_{\gamma,t}u,\tau_{\gamma,t}u)\right|  \leq  2 \sigma \| \dot\gamma(t)\|_{f,\gamma(t)} \| \tau_{\gamma,t}u \|_{f,\gamma(t)}^2 \nonumber \\
	& \leq & \frac{2\sigma \|\dot \gamma(0)\|_{f,\gamma(0)}}{1- \sigma t \|\dot\gamma(0)\|_{f,\gamma(0)}} \psi(t).
\end{eqnarray*}
Integrating $(\log \psi(t))' = \psi'(t)/\psi(t)$ from $0$ to $1$, we have
\[
2 \log (1- \sigma\|\dot\gamma(0)\|_{f,\gamma(0)}) \leq \log \frac{\psi(1)}{\psi(0)} \leq - 2 \log (1- \sigma\|\dot\gamma(0)\|_{f,\gamma(0)}),
\]
which means (\ref{eqn:SC_Hf}). The latter inequality (\ref{eqn:SC_Hf_integral}) is obtained by integrating 
\[
(1- \sigma t\|\dot \gamma(0)\|_{f,\gamma(0)})^2 Hf_{\gamma(0)} \preceq \tau_{\gamma,t}^* Hf_{\gamma(t)}	\preceq 
\frac{1}{(1- \sigma t \|\dot\gamma(0)\|_{f,\gamma(0)})^2}
Hf_{\gamma(0)}
\]
from $0$ to $1$.
\end{proof}

An equivalent formulation of (\ref{eqn:SC_Hf}) is 
\begin{equation}\label{eqn:SC_norm}
	1- \sigma \|\dot \gamma(0)\|_{f,\gamma(0)} \leq \frac{\|\tau_{\gamma,1}u \|_{f,\gamma(1)}}{\|u\|_{f,\gamma(0)}}  	\leq 
	\frac{1}{1- \sigma \|\dot\gamma(0)\|_{f,\gamma(0)}} \quad (u \in T_{\gamma(0)}).
\end{equation}
Namely, the parameter $\sigma$ controls variance of 
the local norm $\|\cdot \|_{f}$ under 
the parallel transport with respect to $\nabla$. 
This is in fact equivalent to SC; 
the book by Renegar~\cite{Renegar01} adopts (\ref{eqn:SC_norm}) as the definition of SC 
in the Euclidean case. Indeed, by a similar calculation in \cite[p. 61]{Renegar01}, we can rewrite (\ref{eqn:SC_norm}) as
\begin{eqnarray*}
	&& -2 \sigma \|u\|_{f,\gamma(0)}^2 \|v\|_{f,\gamma(0)} + \sigma^2 t \|v\|^2_{f,\gamma(0)} \leq \\
	&& \quad \quad \frac{\|\tau_{\gamma, t} u\|_{f,\gamma(t)}^2 - \|u\|_{f,\gamma(0)}^2}
 	{t} \leq \frac{2\sigma\|u\|_{f,\gamma(0)}^2\|v\|_{f,\gamma(0)} + \sigma^2 t \|v\|^2_{f,\gamma(0)} }{(1 - \sigma t \|v\|_{f,\gamma(0)})^2},
\end{eqnarray*}  
where $v := \dot\gamma(0)$. For $t \to 0$, we have (\ref{eqn:SC}).
%
%
\paragraph{Information geometric view.}
We finally present a (dual) formulation of the above properties 
in terms of information geometry. 
Define the dual inner product $\langle,\rangle^{*}_{f,x}$ on $T_x^*$ 
by $\langle p,q\rangle^{*}_{f,x} := Hf_x^{-1}(p,q)$.
Accordingly, define the dual norm $\|p\|_{f,x}^* := \sqrt{\langle p,q \rangle^*_{f,x}}$.
\begin{Lem}\label{lem:SC_norm*} Under the same assumption of Proposition~\ref{prop:SC}, it holds
	\begin{equation}
		1- \sigma \|\dot \gamma(0)\|_{f,\gamma(0)} \leq \frac{\|p\|^*_{f,\gamma(1)}}{\|\tau_{\gamma,1}^* p\|^*_{f,\gamma(0)}}\leq 
		\frac{1}{1- \sigma \|\dot\gamma(0)\|_{f,\gamma(0)}} \quad (p \in T_{\gamma(1)}^*).
	\end{equation}
\end{Lem}
This is obtained from inverting (\ref{eqn:SC_Hf}) by (\ref{eqn:tau}) and (\ref{eqn:F<G}).
The covariant derivative (\ref{eqn:nabla_uG}) of covariant vectors
defines a connection $\nabla$ of cotangent bundle $T^*(M) = \bigcup_{x \in M} T^*_x$, 
where the parallel transport is given by $(\tau^{*}_{\gamma,t})^{-1}: T^*_{\gamma(0)} \to T^*_{\gamma(t)}$.
The Hessian $Hf$ is a vector bundle isomorphism $T(M) \to T^*(M)$.
Then, the dual connection $\nabla^*$ (Remark~\ref{rem:information_geometry}) 
is the pull back of the connection $\nabla$ on $T^*(M)$ by $Hf$; 
see \cite[Proposition 5.1]{NomizuSasaki}.
Specifically, $\nabla^*$ is given by
\begin{equation}
	\nabla^*_u X = Hf^{-1}_x \nabla_u Hf(X) \quad (u \in T_x, x \in M),
\end{equation}
where $Hf(X) = (Hf_x(X_x, \cdot))_{x \in M}$ is a $(0,1)$-tensor field.
The corresponding parallel transport $\delta_{\gamma,t}:T_{\gamma(0)} \to T_{\gamma(t)}$
for a curve $\gamma:[0,t] \to M$ is given by
\begin{equation}
	\delta_{\gamma,t}(u) = Hf_{\gamma(1)}^{-1} (\tau^*_{\gamma,t})^{-1} Hf_{\gamma(0)}(u,\cdot) \quad (u \in T_{\gamma(0)}).
\end{equation}
Consider to transport a vector $u \in T_{\gamma(0)}$ to $T_{\gamma(t)}$ 
along curve $\gamma$ by means of $\nabla$, and then return it back 
to $T_{\gamma(0)}$ along $\gamma$ by means of $\nabla^*$.
This gives rise to an automorphism 
$\delta_{\gamma,t}^{-1} \tau_{\gamma,t}$ on $T_{\gamma(0)}$, which is 
written as
\begin{equation}
	\delta_{\gamma,t}^{-1} \tau_{\gamma,t} u = Hf_{\gamma(0)}^{-1} Hf_{\gamma(t)} (\tau_{\gamma,t} u, \tau_{\gamma,t}(\cdot)) \quad (u \in T_{\gamma(0)}).
\end{equation}
Particularly, $\delta_{\gamma,t}^{-1} \tau_{\gamma,t}$ is 
a symmetric automorphism with respect to $\langle ,\rangle_{f,x}$. 
Then, via (\ref{eqn:G^-1F}), Proposition~\ref{prop:SC} is rephrased as follows.
\begin{Prop}\label{prop:information_geometry}
	Under the same assumption of Proposition~\ref{prop:SC}, we have the following: 
	\begin{itemize}
		\item[(1)] Each eigenvalue $\mu$ of $\delta_{\gamma,1}^{-1} \tau_{\gamma,1}$ 
		satisfies
		\[
		(1-\sigma \|\dot\gamma(0)\|_{f,x})^2 \leq \mu \leq  \frac{1}{(1-\sigma \|\dot\gamma(0)\|_{f,x})^2}.		
		\] 
		Additionally, each eigenvalue $\mu'$ of $\int_{0}^{1}\delta_{\gamma,t}^{-1} \tau_{\gamma,t} {\rm d}t$ satisfies 
			\[
		1-\sigma \|\dot\gamma(0)\|_{f,x}+ \frac{1}{3} \sigma^2 \|\dot\gamma(0)\|_{f,x}^2 \leq \mu' \leq \frac{1}{1-\sigma \|\dot\gamma(0)\|_{f,x}}.		
		\]
		\item[(2)] It holds
			\begin{equation}
			1- \sigma \|\dot \gamma(0)\|_{f,\gamma(0)} \leq \frac{\|v\|_{f,\gamma(1)}}{\|\delta_{\gamma,1}^{-1} v\|_{f,\gamma(0)}}\leq 
			\frac{1}{1- \sigma \|\dot\gamma(0)\|_{f,\gamma(0)}} \quad (v \in T_{\gamma(1)}).
		\end{equation}
	\end{itemize}
\end{Prop}
The second statement follows from Lemma~\ref{lem:SC_norm*} 
and $\| \tau_{\gamma,1}^*p\|^*_{f,\gamma(0)} = \|\delta^{-1}_{\gamma,1}v\|_{f,\gamma(0)}$ 
for $p = Hf_{\gamma(1)}v$.

%

\subsection{Newton's method}
Here we consider the Newton's method on $M$; 
see \cite[Section 6.2]{Boumal23} for a manifold version of Newton's method.
Based on the second order approximation 
$u \mapsto f(x) + Df_x(u) + \frac{1}{2} Hf_x(u,u)$ of $u \mapsto f (\exp_x u)$, 
the {\em Newton direction} at $x \in C$
is defined as the vector  $\overline{u} \in T_x$ satisfying
\begin{equation}
	Df_{x}(\cdot) + Hf_x(\overline{u},\cdot) = 0.
\end{equation}
Equivalently, $\overline{u} := - Hf_x^{-1}(Df_x)$, 
where $Hf_x$ is viewed as $T_{x} \to T_x^*$. 
Then the Newton update at $x$ is formulated as 
\begin{equation}
x \leftarrow x^+ := \exp_{x}\overline{u}.
\end{equation}
The {\em Newton decrement} $\lambda_f(x)$ is defined by
\begin{equation}
	\lambda_f(x) := \|\overline{u}\|_{f,x}  = \|Df_{x}\|^*_{f,x}. 
\end{equation}
For SC functions, 
the Newton decrement 
provides a  certificate for the current point $x$ to belong to 
the region of quadratic convergence of the Newton iterations.
The following corresponds to \cite[Theorem 4.1.14]{Nesterov04} 
(\cite[Theorem 5.2.2.1]{Nesterov18}), 
\begin{Thm}\label{thm:Newton}
	Suppose that $f$ is $\sigma$-SC.
	For $x \in C$, if $\lambda_f(x) < 1/\sigma$, then $x^+ \in C$ and
	\begin{equation}\label{eqn:quadratic}
	\lambda_f(x^+) \leq \frac{\sigma \lambda_f(x)^2}{(1- \sigma\lambda_f(x))^2}.
	\end{equation}
	In particular, if $\lambda_f(x) < (3- \sqrt{5})/2\sigma = 0.3819\cdots /\sigma$, 
	then the point $x$ belongs to 
	the region of quadratic convergence of the Newton iteration.
\end{Thm}
The proof basically adapts \cite[Therem 4.1.14]{Nesterov04}, 
although it may be less obvious than others and has a different presentation for 
revealing its information geometric nature.
The following formula of the next Newton direction by 
parallel transports $\delta,\tau$ of $\overline{u}$
may be interesting in its own right, 
though it is an obvious rewriting. 
\begin{Lem}
	Suppose that $x^+ = \exp_x \overline{u}$ belongs to $C$.
Let $\overline{u}^{+}$ denote the Newton direction at $x^+$.	
Then,
\begin{equation}\label{eqn:Newton}
\delta_{\gamma,1}^{-1} \overline{u}^{+} = \overline{u} - \int_{0}^{1} \delta_{\gamma,t}^{-1} \tau_{\gamma,t} \overline{u} {\rm d}t,
\end{equation}
where $\gamma(t) := \exp_{x}t\overline{u}$.
\end{Lem}
\begin{proof}
	From $\overline{u}^+ = - Hf_{\gamma(1)}^{-1} Df_{\gamma(1)}$, 
	it holds $\delta_{\gamma,1}^{-1} \overline{u}^{+} = - Hf_{\gamma(0)}^{-1} \tau^{*}_{\gamma,1} Hf_{\gamma(1)} Hf_{\gamma(1)}^{-1} Df_{\gamma(1)} = - Hf_{\gamma(0)}^{-1} \tau^{*}_{\gamma,1}  Df_{\gamma(1)}$. Here $\tau^{*}_{\gamma,1}  Df_{\gamma(1)}$ is written as
	\[
	\tau_{\gamma,1}^* Df_{x^+}  =  Df_{x}(\cdot)  + \int_{0}^{1} \frac{\rm d}{{\rm d}t}  Df_{\gamma(t)}(\tau_{\gamma,t}(\cdot)) {\rm d}t   = - Hf_x(\overline{u},\cdot) + \int_{0}^{1} Hf_{\gamma(t)}(\tau_{\gamma,t}\overline{u}, \tau_{\gamma,t}(\cdot))  {\rm d}t.
	\]
	Applying $- Hf_{\gamma(0)}^{-1}$, we have (\ref{eqn:Newton}).
\end{proof}

\begin{proof}[Proof of Theorem~\ref{thm:Newton}]
	Let $\lambda := \lambda_f(x)$.
	Then $x^+\in C$ follows from Proposition~\ref{prop:SCB1} (1) and 
	$\lambda_f(x) = \|\overline{u}\| < 1/\sigma$.
	By Proposition~\ref{prop:information_geometry} (2), we have
	\begin{equation}\label{eqn:lambda+}
		\lambda_f(x^+) = \|\overline{u}^+\|_{f,x^+} \leq \frac{1}{1- \sigma \lambda} \|\delta_{\gamma,1}^{-1} \overline{u}^+ \|_{f,x}.
	\end{equation}
	Let $K: T_{x} \to T_x$ denote the symmetric linear transformation 
	defined by 
	\[
	K u :=  u - \int_{0}^{1} \delta_{\gamma,t}^{-1} \tau_{\gamma,t} u {\rm d}t.
	\]
	By the above lemma, 
	it holds $\delta_{\gamma,1}^{-1} \overline{u}^+  = K\overline{u}$.
	By Proposition~\ref{prop:information_geometry} (1) (and $\sigma \lambda < 1$), the absolute value 
	of each eigenvalue of $K$ is bounded above by $(1- \sigma \lambda)^{-1} - 1 = \sigma \lambda/(1- \sigma\lambda)$. Hence we have
	\begin{equation}\label{eqn:norm}
		\|\delta_{\gamma,1}^{-1} \overline{u}^+ \|_{f,x} = \| K \overline u\|_{f,x} \leq \|K \|_{\rm op} \|\overline{u} \|_{f,x} \leq 
		\frac{\sigma \lambda}{1- \sigma \lambda}\lambda,
	\end{equation}
where $\|K \|_{\rm op} := \sup_{u \in T_x, u \neq 0}\|K u\|_{f,x}/\|u\|_{f,x}$ denotes the operator norm of $K$. 
By (\ref{eqn:lambda+}) and (\ref{eqn:norm}), we have (\ref{eqn:quadratic}).	
\end{proof}
Consider a scaled variant ({\em damped Newton update}) of the Newton update
\begin{equation}
x \leftarrow x^+_{\rm damp} := \exp_{x}\frac{1}{1+ \sigma\lambda_f(x)}\overline{u}.
\end{equation}
This 
guarantees global convergence, as in \cite[Theorem 4.1.12]{Nesterov04}(\cite[Theorem 5.1.14]{Nesterov18}). 
\begin{Prop}
	Suppose that $f$ is $\sigma$-WSC. It holds $x^+_{\rm damp} \in C$ and
	\begin{equation}
	f(x^+_{\rm damp}) \leq f(x) - \frac{1}{\sigma^2} \omega(\sigma \lambda_f(x)).
	\end{equation}
	Thus, for given $x_0 \in M$,  
	the damped Newton update
	produces a point $x$  with $\lambda_f(x) < \beta/\sigma$ in $O(\sigma^2(f(x_0)-f^*)/\omega(\beta))$ iterations. 
\end{Prop}
\begin{proof}
	Let $\lambda := \lambda_f(x)$. 
	By $\|\overline{u}/(1+ \sigma \lambda)\|_{f,x} = \lambda/(1+ \sigma \lambda) < 1/\sigma$ and Proposition~\ref{prop:WSC1}~(1), it holds $x_{\rm damp}^+ \in C$. By Proposition~\ref{prop:WSC2}, we have
	\begin{eqnarray*}
		f(x^+_{\rm damp}) &\leq& f(x) + \frac{1}{1+ \sigma \lambda}Df_x(\overline{u}) + \frac{1}{\sigma^2} \omega_* \left(\sigma \frac{\lambda}{1+ \sigma \lambda}\right) \\
		&=&  f(x) - \frac{\lambda^2}{1+ \sigma \lambda} + \frac{1}{\sigma^2} \omega_*\left(\sigma \frac{\lambda}{1+ \sigma \lambda}\right)  = f(x) - \frac{1}{\sigma^2} \omega(\sigma \lambda).
	\end{eqnarray*}
\end{proof}

\subsection{Self-concordant barrier}

For $\nu > 0$, 
a {\em $\nu$-self-concordant barrier ($\nu$-SCB)} 
of (the closure of) $C$ is a $1$-SC function $F:C \to \RR$ satisfying
\begin{equation}
	\sup_{u \in T_x} 2 DF(u) - HF(u,u) \leq \nu \quad (x \in M).
\end{equation}
By nondegeneracy of $F$,  
this condition is equivalent to
\begin{equation}\label{eqn:SCB}
	 DF(u)^2 \leq \nu HF(u,u) \quad (u \in T_x, x \in M).
\end{equation}
The following corresponds to \cite[Theorem 4.2.2]{Nesterov04} (\cite[Theorem 5.3.2]{Nesterov18}).
\begin{Lem}\label{lem:sum_SCB}
	Let $F_i:C_i \to \RR$ be a $\nu_i$-SCB for $i=1,2$.
	If $C_1 \cap C_2$ has nonempty interior, then $F_1 + F_2: C_1 \cap C_2 \to \RR$
	is a $(\nu_1+\nu_2)$-SCB.
\end{Lem}
The proof is the same and is omitted.
The next corresponds to \cite[Theorem 4.2.4.1]{Nesterov04} (\cite[Theorem 5.3.7.1]{Nesterov18}).
\begin{Prop}\label{prop:SCB1}
	Let $F:C \to \RR$ be a $\nu$-SCB. For any geodesic $\gamma:[0,1] \to C$, it holds
	\begin{equation}\label{eqn:SCB_DF0}
		DF_{\gamma(0)}(\dot\gamma(0)) < \nu.
	\end{equation}
\end{Prop}
\begin{proof}
	We can assume that $DF_{\gamma(0)}(\dot \gamma(0)) > 0$. 
	Define $\phi:[0,1] \to \RR$ by $t \mapsto  DF_{\gamma(t)} (\dot \gamma(t))$.
	By (\ref{eqn:SCB}), it holds
	\[
	\phi'(t) = HF_{\gamma(t)}(\dot \gamma(t),\dot \gamma(t))\geq (1/\nu) \phi(t)^2.
	\]
	In particular, $\phi$ is monotone increasing, and positive. 
	Integrating $\phi'/\phi^2 \geq 1/\nu$ from $0$ to $1$, we have
	\[
	-1/\phi(1)+1/\phi(0) \geq 1/\nu.
	\]
	Then $\phi(0) < \nu$, and we have (\ref{eqn:SCB_DF0}).      
\end{proof}
Instead of formulating other properties of SCB, 
we directly describe the path-following method for 
minimization of an affine function on a compact convex set.
Here an {\em affine function} is a function $\ell$ such that 
it is an affine function on every geodesic $\gamma$, i.e., $(\ell \circ \gamma)''(t) = 0$.
In a general manifold, 
nonconstant affine functions are very rare. 
The main situation we think of is the minimization of 
$(x,y) \mapsto c^{\top}y$ over $C \subseteq M \times \RR^k$; see Example~\ref{ex:MEB}.

Let $Q \subseteq M$ be a compact convex set such that its interior $Q^{\circ}$ is nonempty.
For an affine function $\ell:Q \to \RR$, consider the following optimization problem:
\begin{equation}
	\mbox{Min.} \quad \ell(x) \quad {\rm s.t.}\quad x \in Q.
\end{equation}
We can assume that $\ell$ is not a constant function; then $D\ell_x \neq 0$ for all $x \in Q^\circ$.
Let $\ell^*$ denote the minimum of $\ell(x)$ over $x \in Q$.

Let $F:Q^{\circ} \to \RR$ be a $\nu$-SCB of $Q$. Consider the parametric problem with $t \geq 0$:
\begin{equation}
	\mbox{Min.} \quad f_t(x) := t \ell(x) + F(x)  \quad {\rm s.t.}\quad x \in Q^{\circ}.
\end{equation}
Let $x^*(t)$ denote the minimizer of $f_t$.
As in Euclidean case, 
the curve $t \mapsto x^*(t)$ is called the {\em central path}, 
and the starting point $x^*(0)$, the minimizer of $F$, is called the {\em analytic center} 
of $Q$ (relative to $F$). 
The {\em path-following method} can be formulated as in \cite[Section 4.2.4]{Nesterov04}:
\begin{description}
	\item[Path-following method] 
	\item[Input:] A point $x_0 \in Q^{\circ}$ with $\lambda_F(x_0) \leq \beta$, where $0 < \beta < \frac{\sqrt{\beta}}{1+\sqrt{\beta}}$; say $\beta := 1/9$.
	\item[0:] $(x,t) \leftarrow (x_0,0)$.
	\item[1:] $t^+ := t  + \alpha /\|D\ell_x\|^*_{F,x}$, 
	where $\alpha := \frac{\sqrt{\beta}}{1+ \sqrt{\beta}}- \beta >0$.
	\item[2:] $x^+ := \exp_x \overline{u}$, where $\overline{u}$ is the Newton direction of $f_{t^+}$ at $x$.
	\item [3:] $(x,t) \leftarrow (x^+,t^+)$; go to {\bf 1}.
\end{description}
The initial point $x_0$ can be obtained 
by the damped Newton method.
%
The polynomial iteration complexity is obtained as in the Euclidean case:
\begin{Thm} Suppose that $\|(Df_t)_x\|^*_{F,x} \leq \beta$. Then it holds
	\begin{itemize}
		\item[(1)]  $\|(Df_{t^+})_{x^+}\|^*_{F,x^+} \leq \beta$.
	\end{itemize}
In addition, if $t > 0$, then it holds:
\begin{itemize}
		\item[(2)] $\|D\ell_x\|^*_{F,x} \leq \frac{1}{t} (\beta+ \sqrt{\nu})$; then $t^+ \geq \left(1+ \frac{\gamma}{\beta+ \sqrt{\nu}}\right)t$.
		\item[(3)] $\ell(x) - \ell^* \leq \frac{1}{t}\left( \nu+ \frac{(\beta+ \sqrt{\nu})\beta}{1-\beta}\right)$.
	\end{itemize}
Thus, given  $\epsilon > 0$, 
the path-following method produces
a point $x \in Q^\circ$ with $\ell(x)- \ell^* \leq \epsilon$ in
$O(\sqrt{\nu} \log \nu \|D\ell_{x_0}\|^*_{F,x_0}/\epsilon)$ 
iterations.
\end{Thm}
\begin{proof}
		(1).  Let $\lambda_0 := \| (Df_t)_x \|_{F,x}^{*} \leq \beta$, 
	$\lambda_1 :=  \| (Df_{t^+})_x \|_{F,x}^{*}$ and 	
	$\lambda^+ :=  \| (Df_{t^+})_{x^+} \|_{F,{x^+}}^{*}$.
	Then 
	\[
	\lambda_1 =  \| \alpha D\ell_x/\|D\ell_x \|_{F,x}^* +(Df_t)_x\|_{F,x}^{*} \leq \alpha + \|(Df_t)_x\|_{F,x}^{*} \leq \alpha + \beta = 
	\frac{\sqrt{\beta}}{1+\sqrt{\beta}} < 1.
	\]
	Here $f_t$ is $1$-SC for $t \geq 0$. 
	By Theorem~\ref{thm:Newton} (with $\| (Df_{t^+})_x\|^*_{f_{t^+},x} = \lambda_1 < 1$), it holds
	\[
	\lambda^+ \leq \left(\frac{\lambda_1}{1- \lambda_1}\right)^2 \leq \left(\frac{\beta+\alpha}{1-\beta -\alpha}\right)^2 = \beta.
	\]
	
	(2). $\|D\ell_x\|^*_{F,x} = \frac{1}{t}\|(Df_t)_x - DF_x \|^*_{F,x} 
	\leq \frac{1}{t}(\|(Df_t)_x\|_{F,x}^* + \|DF_x \|^*_{F,x}) \leq \frac{1}{t} (\beta + \sqrt{\nu})$.
	
	(3). Let $x^*$ be a minimizer of $\ell$ over $Q$, 
	and $x^*(t)$ be the unique minimizer of $f_t$ over $Q^\circ$.
	By the assumption of $(M,\nabla)$, 
	we can choose a geodesic $\gamma:[0,1] \to Q$ from $x^*(t)$ to $x^*$.
	Since $\ell$ is an affine function and $(Df_t)_{x^*(t)} = t D\ell_{x^*(t)} + DF_{x^*(t)} = 0$, we have 
	\begin{equation}\label{eqn:add1}
			\ell(x^*(t)) - \ell(x^*) = - D\ell_{x^*(t)}(\dot \gamma(0)) = \frac{1}{t} DF_{x^*(t)}(\dot \gamma(0)) \leq \frac{\nu}{t},
	\end{equation} 
     where the last inequality follows from Proposition~\ref{prop:SCB1}. 
    Also, for geodesic $c: [0,1] \to Q$ connecting $c(0) = x$ and $c(1) = x^*(t)$, we have 
	 \begin{equation}\label{eqn:add2}
	 	t (\ell(x) - \ell(x^{*}(t))) =  - t D\ell_{x}(\dot c(0)) \leq t\|D\ell_x\|^*_{F,x}\|\dot c(0)\|_{F,x}  \leq (\beta+ \sqrt{\nu})\frac{\beta}{1- \beta}. 
	 \end{equation}
For the last inequality, we use $\|\dot c(0)\|_{F,x} \leq \beta/(1-\beta)$, 
which can be seen as follows. By (\ref{eqn:WSC2_Df=>}) with $(Df_t)_{x^*(t)} = 0$ and Cauchy-Schwarz (\ref{eqn:CM}), it holds
\[
\frac{\|\dot c(0)\|^2_{F,x}}{1+ \|\dot c(0)\|_{F,x}} \leq - (Df_t)_{x}(\dot c(0)) \leq \|(Df_t)_{x}\|^*_{F,x} \|\dot c(0)\|_{F,x} \leq \beta \|\dot c(0)\|_{F,x}. 
\]
By (\ref{eqn:add1}) and (\ref{eqn:add2}), 
we have (1).  
\end{proof}
%

We finally examine construction of the logarithmic barrier.
For a convex function $f: M \to \RR$ 
and $\beta \in \RR$, the level set $L_{f, \beta} \subseteq M$ is defined by
\begin{equation}
	L_{f,\beta} := \{ x \in M \mid f(x) \leq \beta \}.
\end{equation}
Denote by $L_{f,\beta}^{\circ}$ its interior.
Define the logarithmic barrier $F_{\beta}: L_{f,\beta}^{\circ} \to \RR$ 
by 
\begin{equation}
	F_{\beta}(x) := - \log (\beta - f(x)) \quad (x \in L_{f,\beta}^{\circ}).
\end{equation}
\begin{Lem}
	$F= F_{\beta}$ is closed convex, and satisfies
	\begin{equation}\label{eqn:SCB1}
		DF(u)^2 \leq HF(u,u) \quad (u \in T_x, x \in L_{f,\beta}^{\circ}).
	\end{equation}
\end{Lem}
\begin{proof}
   Let $\omega  := \beta - f(x) > 0$. Then we have
	\begin{equation}\label{eqn:DF_HF}
	DF(u) = \frac{Df(u)}{\omega},\quad 	HF(u,u) = \frac{Hf(u,u)}{\omega} + \frac{Df(u)^2}{\omega^2}.
	\end{equation}
	Then $HF(u,u) \geq 0$ and hence $F$ is convex.
	Also $HF(u,u) \geq DF(u)^2$ holds.
\end{proof}
If a $\sigma$-SC function $F$ for $\sigma > 0$ satisfies (\ref{eqn:SCB1}),  then $\sigma^2 F$ is a $\sigma^2$-SCB.
The following is a correspondent of \cite[Theorem 5.1.4]{Nesterov18}. 
\begin{Prop}\label{prop:F_beta}
	Suppose that $f:M \to \RR$ is $\sigma$-SC.
	Then $F_{\beta}: L_{f,\beta}^{\circ} \to \RR$ 
	is $\tilde \sigma$-SC for
	\[
	\tilde \sigma := \sqrt{(\sigma (\beta - f^*)^{1/2} + 1)^2+ 9/4},
	\]
	where $f^* := \inf_{x \in M} f (x)$.  In particular, 
	$L_{f,\beta}$ admits an $O(\sigma^2 (\beta -f^*))$-SCB.
\end{Prop}
In the case of WSC,  
the constant $\tilde\sigma$ can be taken as $\sqrt{\sigma^2(\beta-f^*)+1}$, 
since precisely the same proof of \cite[Theorem 5.1.4]{Nesterov18} works.
For the SC case, as in the proof of Lemma~\ref{lem:sum}, a little modification is required, 
which causes the increase of the constant.
\begin{proof}
From (\ref{eqn:DF_HF}), $Df(u)^2 = (Df \otimes Df) (u,u)$, 
and $\nabla_v  (Df \otimes Df) (u,u) = (\nabla_v Df \otimes Df + Df \otimes \nabla_v Df)(u,u) = 2 Df(u) Hf(u,v)$ (see \cite[II. Proposition 1.3 (1)]{Sakai96}), 
the covariant derivative of $HF$ is given by
\begin{equation}\label{eqn:nablaHF}
\nabla_v HF(u,u) = \frac{\nabla_v Hf(u,u)}{\omega} + \frac{Df(v)Hf(u,u)}{\omega^2} + \frac{2Df(u) Hf(u,v) }{\omega^2} + \frac{2Df(v)Df(u)^2}{\omega^3}.
\end{equation}
Hence we have
	\begin{eqnarray*}
		|\nabla_v HF(u,u)|  &\leq&  \frac{2\sigma Hf(v,v)^{1/2} Hf(u,u)}{\omega} + \frac{|Df(v)| Hf(u,u) }{\omega^2} \\ 
		&&  +\frac{2|Df(u)| Hf(v,v)^{1/2} Hf(u,u)^{1/2} }{\omega^2} 
		+  \frac{2|Df(v)| Df(u)^2 }{\omega^3}.
	\end{eqnarray*}
Define $\tau_1,\tau, \xi_1, \xi$ by
\begin{equation*}
	\tau_1 := (Hf(v,v)/\omega)^{1/2},\ \tau := (Hf(u,u)/\omega)^{1/2},\ \xi_1 := |Df(v)|/\omega,\ \xi := |Df(u)|/\omega.
\end{equation*} 
Then we have
\begin{equation}\label{eqn:upperbound}
\frac{|\nabla_v HF(u,u)|}{2 HF(v,v)^{1/2}HF(u,u)} \leq 
\frac{\sigma \omega^{1/2} \tau_1 \tau^2 + (1/2)\xi_1 \tau^2 + \xi \tau_1 \tau + \xi_1 \xi^2}{(\tau_1^2 + \xi_1^2)^{1/2}(\tau^2 + \xi^2)}.
\end{equation}
We bound the RHS.
By homogeneity,
we may consider the optimization problem:
\[
\mbox{Max.} \quad \sigma \omega^{1/2} \tau_1 \tau^2 + (1/2) \xi_1 \tau^2+ \xi \tau_1 \tau + \xi_1 \xi^2	\quad \mbox{s.t.}\quad  \tau_1^2 + \xi_1^2 = 1,\ \tau^2 + \xi^2 = 1.
\]
Fixing $(\tau, \xi)$, 
optimize $(\tau_1,\xi_1)$, which is a linear optimization over the unit circle. 
The optimum is $(\sigma \omega^{1/2} \tau^2 + \xi \tau, (1/2)\tau^2+ \xi^2)/\sqrt{(\sigma \omega^{1/2} \tau^2 + \xi \tau)^2+ ((1/2)\tau^2 + \xi^2)^2}$. 
Then we have
\[
	\mbox{Max.} \quad \sqrt{(\sigma \omega^{1/2} \tau^2 + \xi \tau)^2 + ((1/2)\tau^2+\xi^2)^2}
	\quad \mbox{s.t.} \quad \tau^2 + \xi^2 = 1.
\]
Therefore, the RHS of (\ref{eqn:upperbound}) is at most $\sqrt{(\sigma\omega^{1/2} + 1)^2+ ((1/2)+1)^2}$.
\end{proof}

\section{Squared distance function}\label{sec:distance}
Let $M$ be a Riemannian manifold,  where the inner product and norm on $T_x$ are denoted by
$\langle, \rangle_x$ and $\|\cdot \|_x$, respectively. 
The length of a curve $\gamma:[a,b] \to M$ is defined 
as $\int_{a}^{b} \|\dot \gamma(t)\|_{\gamma(t)} dt$.
The distance $d(x,y)$ between two points $x,y$ is defined as 
the infimum of the length of a curve connecting $x,y$.
We consider the Levi-Civita connection $\nabla$.
Then a minimum-length curve is a geodesic with respect to $\nabla$.
For a $2$-dimensional subspace $F$ in $T_x$,
the {\it sectional curvature} of $F$ is defined 
as $\langle R(u,v)v, u \rangle_x$ where $\{u,v\}$ 
is an orthonormal basis of $F$.

Suppose that $M$ is an {\em Hadamard manifold}, i.e., it is complete, simply-connected, 
and has nonpositive sectional curvature for every tangent plane at every point. 
Then  
a geodesic connecting any two points is uniquely determined, 
up to the choice of an affine parameter.
The exponential map $u \mapsto \exp_x u$ gives a diffeomorphism from $T_x$ to $M$.
The distance $d(x,y)$ coincides with the length of the unique geodesic between $x$ and $y$.
See \cite[VI. Theorem 4.1]{Sakai96}.

Fix an arbitrary point $p \in M$. 
Define the distance and squared distance functions $d_p, f_p: M \to \RR$ by 
\[
	d_p(x) := d(p,x), \quad f_p(x) := \frac{1}{2} d(p,x)^2 \quad (x \in M).
\]
It is known \cite[Lemma 12.15]{Lee18} 
that $f_p$ is smooth and $1$-strongly convex 
(i.e.,$Hf_p(u,u) \geq \|u\|^2$ for $u \in T_x, x\in M$), and hence nondegenerate.
Note that $d_p$ is also convex but nondifferentiable at $p$ and degenerate.

For exploring a nontrivial class of SC-functions, 
the first trial might be to examine whether $f_p$ is SC or not.
In this section, we show that $f_p$ is SC for hyperbolic spaces. 
A {\em hyperbolic space} of curvature $-\kappa < 0$ is an Hadamard manifold 
such that every sectional curvature of every point is constant $- \kappa$.
The main result in this section is the following.
\begin{Thm}\label{thm:SC_hyperbolic}
	Suppose that $M$ is a hyperbolic space with curvature $- \kappa < 0$, and let $p \in M$.
	\begin{itemize}
		\item[(1)] $f_p$ is $\sqrt{\kappa}/2$-SC, and is not $(\sqrt{\kappa}/2- \epsilon)$-SC for every small
		$\epsilon > 0$.
		\item[(2)] $f_p$ is $\sqrt{4\kappa/27}$-WSC, and is not $(\sqrt{4\kappa/27}- \epsilon)$-WSC for every small
		$\epsilon > 0$.
\end{itemize}
\end{Thm}
Note that 
in the hyperboloid model ($\kappa =  1$), 
the $\sqrt{4/27}$-WSC property of $f_p$ is claimed in 
\cite[Lemma 4]{JMM08} (without proof).
It seems reasonable to conjecture that $f_p$ is 
$\sqrt{\kappa}/2$-SC for any Hadamard manifold having 
curvature lower bound $-\kappa$; it may be proved via comparison theory.

As a consequence of Theorem~\ref{thm:SC_hyperbolic} and Proposition~\ref{prop:SCB1},
any ball in a hyperbolic space admits an SCB.
For $R > 0$, let $B(p;R)^{\circ} := \{ x \in M \mid d(p,x) < R\}$ 
be the open ball with center $p$ and radius $R$. 
let $F_{p,R}: B(p;R)^{\circ} \to \RR$ be defined by
\begin{equation}
	F_{p,R}(x) := - \log \left(R^2- d(p,x)^2\right) \quad (x \in M).
\end{equation}
\begin{Cor}
	Suppose that $M$ is a hyperbolic space with curvature $- \kappa < 0$.
	Then $F_{p,R}$ is $O(\sqrt{\kappa} R)$-SC. 
	In particular, the open ball $B(p;R)^{\circ}$ admits an $O(\kappa R^2)$-SCB.
\end{Cor}

We show that the order $O(\sqrt{\kappa} R)$ cannot be improved. 
Define the function $R \mapsto \sigma_{\kappa} (R)$ as
the infimum of $\sigma$ for which $F_{p,R}$ is $\sigma$-SC.
\begin{Prop}\label{prop:tightness}
		Suppose that $M$ is a hyperbolic space with curvature $- \kappa < 0$.
	   Then $\sigma_{\kappa}(R) = \Omega (\sqrt{\kappa} R)$.
\end{Prop}
This may be a pity result.
In the path-following method, 
for guaranteeing the existence of the analytic center, 
we need to add a constraint to make the feasible region bounded.
A natural constraint is the distance constraint $d(p,x) \leq R$.
Then, adding the above barrier increases the barrier parameter 
$\nu$ by $\Omega (\kappa R^2)$, and 
the resulting iteration complexity of the path-following method depends on $\sqrt{\kappa} R$.
 
In the Euclidean case, the logarithmic barrier of $f_p$ provides a $2$-SCB for a ball of arbitrary diameter.
Further, a general result of Nesterov and Nemirovskii~\cite{NN94}
says that every bounded convex set in $\RR^n$ admits an $O(n)$-SCB.
The above result may suggest the nonexistence of such an SCB with constant 
independent of the size of diameter.

We provide one example to which the presented results are applicable.
\begin{Ex}[{Minimum enclosing ball in hyperbolic space~\cite{NH15}}]\label{ex:MEB}
	The {\em minimum enclosing ball problem (MEB)}  on a manifold $M$
	is: 
	Given distinct points $p_1,p_2,\ldots,p_m$ in $M$, find the smallest
	ball including them. 
	MEB is a nonsmooth optimization problem:
	\begin{equation}
		\mbox{Min.} \quad r \quad {\rm s.t.}\quad (x,r) \in M \times \RR,\ d(p_i,x) \leq r \ (i=1,2,\ldots,m).
	\end{equation}
	In the case of Euclidean space $M = \RR^n$, 
    MEB is a well-studied problem in computational geometry, and can be formulated 
    by second-order cone program to which an interior-point method is applicable; see e.g., \cite{KMY03}.
    
	We here assume that $M$ is a $d$-dimensional hyperbolic space with
	curvature $-\kappa$. This case is addressed by Nielsen and Hadjeres~\cite{NH15}. 
   We may further assume $\kappa =1$ (see below).
   Let $r^*$ denote the optimal value of MEB. 
   Let $R := 2 \max_{1 \leq  i<j\leq m} d(p_i, p_j)$. Clearly $r^* < R$. 
   So we may add constraint $r^* \leq R$ to make the feasible region bounded.
   For applying our framework, we consider the following equivalent formulation:
   \begin{equation}\label{eqn:MEB'}
   \mbox{Min.} \quad s \quad {\rm s.t.}\quad (x,s) \in M \times \RR,\ d(p_i,x)^2 \leq s \leq R^2 \ (i=1,2,\ldots,m).
   \end{equation}
   Consider the following barrier 
   \begin{equation}
   F(x,s) := - \sum_{i=1}^m K^2 \log (s - d(p_i,x)^2) - \log (R^2-s),
   \end{equation}
    where $K = \sqrt{2^{-1/2}(R+1)^2+9/4}$. By Theorem~\ref{thm:SC_hyperbolic},
    $(s,x) \mapsto d(p_i,x)^2 - s$ is (degenerate) $2^{-1/2}$-SC. 
     The proof of Proposition~\ref{prop:F_beta} implies that 
    $(s,x) \mapsto \log (s - d(p_i,x)^2)$ satisfies (\ref{eqn:SC}) with 
    $\sigma = K$ if $d(p_i,x)^2 \leq s \leq R^2$. 
    The last term of $F$ is a well-known $1$-SCB.
    With Lemmas~\ref{lem:sum} and \ref{lem:sum_SCB},   
    we conclude that $F$ is an $O(mR^2)$-SCB for the problem (\ref{eqn:MEB'}). 
    
    By the damped Newton method with starting point $(p_i, R^2/2)$ (say),  
    we obtain the next starting point $(x_0,s_0)$ in 
    $O(m R^2 \log R)$ iterations. Given $\epsilon > 0$, 
    by the path-following method with starting point $(x_0,s_0)$ and accuracy 
    $\epsilon' := \epsilon / \min_{1 \leq i < j\leq m} d(p_i,p_j)$, 
    we obtain an additive $\epsilon$-approximation solution $(x,\sqrt{s})$ of MEB with 
    $\sqrt{s} - r^* \leq \epsilon'/(\sqrt{s}+r^*) \leq \epsilon$ in $O(m^{1/2} R \log mR\alpha_0 /\epsilon)$ iterations, where $\alpha_0 := \|(0,1)\|^*_{F,s_0}$.  The total number of Newton iterations is 
    $O(mR^2 \log R + m^{1/2} R \log mR\alpha_0/\epsilon)$. 
    Although the dependence of $R$ is disappointing,
     it seems the first $\log 1/ \epsilon$-dependence algorithm to this problem. 
     As far as we know, the only known algorithm having explicit complexity is an  $O(dm/\epsilon^2)$-time 
     $(1+\epsilon)$-approximation due to Nielsen and Hadjeres~\cite{NH15}. 
\end{Ex}

The rest of this section is devoted to proving Theorem~\ref{thm:SC_hyperbolic}
and Proposition~\ref{prop:tightness}. 
Although there are several models of hyperbolic space 
in which explicit computation can be performed, 
we do not take such an approach 
and take a  ``model-free" approach based on Jacobi fields. 

For a geodesic $\gamma:[0,l] \to M$, 
a {\em Jacobi field} along $\gamma$ 
is a vector field $X = (X(t))_{t \in [0,l]}$, where $X(t) \in T_{\gamma(t)}$,
satisfying the Jacobi equation
\begin{equation}\label{eqn:Jacobi}
	\nabla_{\dot{\gamma}(t)}\nabla_{\dot{\gamma}(t)} X(t) + R(X(t), \dot{\gamma}(t))\dot{\gamma}(t) = 0 \quad (t \in [0,l]).
\end{equation}
This is a linear differential equation.
Then, the solution $X(t)$ is uniquely determined 
under the initial values $X(0), \nabla_{\dot{\gamma}(0)}X(0)$, or under boundary values $X(0), X(l)$. 
\begin{Lem}[{see \cite[p.35, 36]{Sakai96}}]\label{lem:Jacobi}
	Let $\alpha:[0,l] \times (- \epsilon, \epsilon) \to M$ be 
	a smooth map such that the curve
	$t \mapsto \alpha(t,s)$
		is a geodesic for each $s \in (-\epsilon,\epsilon)$.
	Then $D\alpha(t,0) (\frac{\partial}{\partial s})$ 
	is a Jacobi field along geodesic $t \mapsto \alpha(t,0)$.
\end{Lem}
The Hessian of $f = f_p$ is given as follows.
Recall distance function $d = d_p: M \to \RR$.
\begin{Lem}[{see \cite[p.108--110]{Sakai96}}]\label{lem:Dd^2Hd^2}
	Let $x \in M$ with $p \neq x$, and let $\gamma:[0,l] \to M$ be a geodesic
	with $\gamma(0) = p$, $\gamma(l) = x$, and $l = d(p,x)$.
	For $u \in T_x$, it holds	
	\begin{itemize}
		\item[(1)] $Dd (u) = \langle \dot{\gamma}(l),u \rangle$.
		\item[(2)] $Df (u) = d(p,x) \langle \dot{\gamma}(l),u \rangle$. 
		\item[(3)] $Hf (u,u) = d(p,x) \langle \nabla_{\dot{\gamma}(l)}X(l),u \rangle$, where $X$ is the Jacobi field along $\gamma$ under the boundary condition
		\[
		X(0) = 0, \quad X(l) = u.
		\]
	\end{itemize}
\end{Lem}

Suppose that $M$ is a hyperbolic space with curvature $- \kappa < 0$.
It is known \cite[II. Lemma 3.3]{Sakai96}
that the curvature tensor $R$ is written as
\begin{equation}\label{eqn:R(X,Y)Z_kappa}
R(X,Y)Z = - \kappa (\langle Y,Z \rangle X - \langle X,Z\rangle Y),
\end{equation}
from which solutions of the Jacobi equation
are explicitly written.

For a positive real $\alpha > 0$,  
consider the change of metric
$\langle, \rangle \to \alpha \langle,\rangle$.
The distance $d$ changes as $d \to \sqrt{\alpha}d$.
The Levi-Civita connection $\nabla$ and curvature tensor $R$ do not change.
The sectional curvature $- \kappa$ changes as $- \kappa \to - \kappa/\alpha$. 
Therefore, by changing metric, we may assume 
\begin{equation}
	\kappa =  1.
\end{equation}
This normalization is for simplicity of the equations we will compute.

\begin{Lem}\label{lem:nablaHf_hyperbolic} Let $x \in M$ with $x \neq p$ and let $\gamma:[0,l] \to M$ be the geodesic from $p$ to $x$ with $l = d(p,x)$. Then it holds
	\begin{eqnarray}
			Hf (u,u) &=& 	(l \coth l) \left(\langle u,u \rangle - \langle u, \dot{\gamma}(l) \rangle^2 \right)  +  \langle u, \dot{\gamma}(l)\rangle^2,  \label{eqn:Hd^2(u,u)}\\
			\nabla_{v} Hf (u,u) &=& \Phi(l) 
		\langle v, \dot \gamma (l) \rangle \left( \langle u, u \rangle - \langle u, \dot\gamma(l)\rangle^2 \right)  \nonumber \\
		&& + 2  \left( l - \Phi(l) \right)
		\langle u, \dot\gamma(l)\rangle \left( \langle v, \dot \gamma (l) \rangle \langle u, \dot\gamma(l)\rangle - \langle u, v \rangle \right),
	\end{eqnarray}
	where $\Phi: \RR \to \RR$ is defined by
	\[
	\Phi(z) := (z \coth z)' = \coth z  + z - z \coth^2 z.
	\]
\end{Lem}
Explicit computation of Jacobi fields and the Hessian $Hf(u,u)$ in a hyperbolic space 
can be found in several textbooks; see e.g., \cite[p. 136, p. 154]{Sakai96}. 
We include it for computation of $\nabla_v Hf(u,u)$.
\begin{proof}
	We explicitly compute the Jacobi field $X = (X(t))_{t \in [0,l]}$ along $\gamma$, 
	under boundary condition $X(0) = 0$, $X(l) = u$, where
	\cite[p.136]{Sakai96} has similar computation we consult. 
	Decompose $X(t)$ as
	\[
	X(t) = \tilde X(t) + b(t) \dot{\gamma}(t),
	\]
	where $b(t) := \langle X(t), \dot{\gamma}(t) \rangle$ and 
	$\tilde X(t) := X(t) - b(t) \dot{\gamma}(t)$.
	Namely $\tilde X(t)$ is the component of $X(t)$ orthogonal to $\dot{\gamma}(t)$.
	By  (\ref{eqn:R(X,Y)Z_kappa}), it holds 
	$R(X(t), \dot\gamma(t))\dot\gamma(t) = 
	-  (\langle \dot\gamma(t),\dot \gamma(t)\rangle X(t)- \langle \dot \gamma(t),X(t)\rangle \dot\gamma(t)) = - \tilde X(t)$, where $\langle \dot\gamma(t),\dot \gamma(t)\rangle = 1$
	by $l = d(p,x)$.
	Parallel transports of $\tilde X(t)$ 
	along $\gamma$ keep orthogonal to the tangent direction of $\gamma$.
	From this fact, 
	the Jacobi equation is decomposed as
	\[
	{b}''(t) = 0,\quad 	\nabla_{\dot{\gamma}(t)}\nabla_{\dot{\gamma}(t)} \tilde X (t) - 
	\tilde X(t) = 0.
	\]
	Represent $\tilde X(t) = \sum_{i} \alpha_i(t) \tau_{\gamma,t} e_i$, 
	where $\{e_i\}$ is a basis of 
	the subspace of $T_{\dot \gamma(0)}$ orthogonal to $\dot \gamma(0)$.
	Then
	the second equation becomes $\alpha_i''(t) - \alpha_i(t) = 0$. Thus we obtain
	\[
	b(t) = b(0) + b'(0)t,\quad  \tilde X(t) = (\cosh t) \tau_{\gamma,t}\tilde X(0) + (\sinh t)  \tau_{\gamma,t} \nabla_{\dot{\gamma}(0)}\tilde X(0).
	\]
	Consider the boundary condition $X(0) = 0$ and $X(l) = u$.
	By  $X(0) = 0$, it must hold  $\tilde X(0) = 0$ and $b(0) = 0$.
	Since $\tau_{\gamma,l}\nabla_{\dot{\gamma}(0)}\tilde X(0)$ is orthogonal to $\dot\gamma(l) = \tau_{\gamma,l} \dot \gamma(0)$,
	$X(l) = b'(0)l \dot\gamma(l) + (\sinh l)  \tau_{\gamma,t} \nabla_{\dot{\gamma}(0)}\tilde X(0) 
	= u$ implies
	\[
	b'(0) = \langle u, \dot\gamma(l) \rangle/l, \quad \tau_{\gamma,l} \nabla_{\dot\gamma(0)}\tilde X(0) = (u- \dot \gamma(l)\langle u, \dot\gamma(l)\rangle )/\sinh l.
	\]
Then $\nabla_{\dot\gamma(l)} X(l)$ is given by
\begin{equation}\label{eqn:nabla_X(l)}
\nabla_{\dot\gamma(l)} X(l) = \dot\gamma(l)\langle u, \dot\gamma(l)  \rangle/l +
\ (\coth l) (u- \dot \gamma(l)\langle u, \dot\gamma(l)\rangle). 
\end{equation}
Apply Lemma~\ref{lem:Dd^2Hd^2}~(3) to obtain (\ref{eqn:Hd^2(u,u)}).

	Consider the geodesic $s \mapsto c(s) := \exp_x s v$.
	Let $\gamma_{s}:[0, l] \to M$ be the geodesic from $p$ to $c(s)$
	(not parametrized by the arc-length). 
	For $s \in (-\epsilon, \epsilon)$,  let $l_s := d(p,c(s))$ and $u_s := \tau_{c,s}u$.
	By (\ref{eqn:Hd^2(u,u)}) with 
	reparametrized geodesic $t \mapsto \gamma_s((l /l_s)t)$ $(t \in [0,l_s])$, 
	we have
	\begin{equation}\label{eqn:u_s}
	(Hf)_{c(s)}(u_s,u_s) = 
	(l_s \coth l_s) \langle u_s,u_s \rangle 
	+  \left(1 - 
	 l_s \coth l_s\right) (l/l_s)^2 \langle u_s, \dot\gamma_s(l)\rangle^2. 
	\end{equation}
	By (\ref{eqn:nabla_uG}), the covariant derivative $\nabla_vHf(u,u)$ is obtained by
	computing $\partial_{s \to 0} := d/ds\mid_{s \to 0}$ of (\ref{eqn:u_s}). We use
	\[
	\partial_{s \to 0} l_s  = \langle \dot \gamma (l), v \rangle, \quad \partial_{s \to 0} \langle u_s, u_s\rangle = 0,\quad \partial_{s \to 0} \langle u_s, \dot \gamma_s(l) \rangle = \langle u, \left. \nabla_{\dot c(s)} \dot \gamma_s(l) \right|_{s \to 0} \rangle,
	\]
	where the first one follows from Lemma~\ref{lem:Dd^2Hd^2} (1) and 
	the other follow from 
	$X \langle Y,Z\rangle = \langle \nabla_X Y,Z\rangle + \langle Y, \nabla_X Z\rangle$ 
	and $\nabla_{\dot c(s)} u_s = 0$.
Hence we have
\begin{eqnarray}
	\nabla_{v} f(u,u) &=& \Phi(l)
	\langle \dot \gamma (l), v \rangle \langle u, u \rangle + \left( - \Phi(l) - 2/l + 2 \coth l \right)	\langle \dot \gamma (l), v \rangle \langle u, \dot\gamma(l)\rangle^2 \nonumber \\
	&&  + 2 \left(1 - 
	l \coth l \right) \langle u, \dot\gamma(l) 
	\rangle \langle u, \left. \nabla_{\dot c(s)} \dot \gamma_s(l) \right|_{s \to 0}  \rangle. \nonumber \\
	&=& \Phi(l)
	\langle \dot \gamma (l), v \rangle (\langle u, u \rangle -  \langle u, \dot \gamma (l) \rangle^2) \nonumber \\
	&&+  2 \left(1 - 
	l \coth l \right) \langle u, \dot\gamma(l) \rangle( \langle u, \left. \nabla_{\dot c(s)} \dot \gamma_s(l) \right|_{s \to 0}\rangle  - \langle v, \dot \gamma (l) \rangle \langle u, \dot\gamma(l)\rangle/l  ). \nonumber
	\label{eqn:nabla_vf(u,u)} 
\end{eqnarray}
	To compute $\nabla_{\dot c(s)} \dot \gamma_s(l) |_{s \to 0}$, 
	consider the (smooth) map $\alpha:[0,l] \times (-\epsilon, \epsilon) \to M$
such that $(t,s) \mapsto \gamma_s(t)$. Let 
$\frac{\partial \alpha}{\partial s}(t,s) := D\alpha (t,s)(\frac{\partial}{\partial t})$ and 
$\frac{\partial \alpha}{\partial t}(t,s) := D\alpha (t,s)(\frac{\partial}{\partial s})$.
Then  $\left. \nabla_{\dot c(s)} \dot \gamma_s(l) \right|_{s \to 0}=
\nabla_{\frac{\partial \alpha}{\partial s}} \frac{\partial \alpha}{\partial t}(l,0) =\nabla_{\frac{\partial \alpha}{\partial t}} \frac{\partial \alpha}{\partial s} (l,0)$, 
since $\nabla_{\frac{\partial \alpha}{\partial s}} \frac{\partial \alpha}{\partial t} = \nabla_{\frac{\partial \alpha}{\partial t}} \frac{\partial \alpha}{\partial s} + [\frac{\partial \alpha}{\partial s}, \frac{\partial \alpha}{\partial t}]$ 
and $[\frac{\partial \alpha}{\partial s}, \frac{\partial \alpha}{\partial t}] = D\alpha ([\frac{\partial}{\partial s}, \frac{\partial}{\partial t}]) = 0$; see \cite[II. Lemma 2.2]{Sakai96}.
By Lemma~\ref{lem:Jacobi},  
$Y(t) := \frac{\partial \alpha}{\partial s} (t,0)$
is a Jacobi field along the geodesic $\gamma$, and satisfies 
$Y(0) = 0$ and $Y(l) = v$. 
By (\ref{eqn:nabla_X(l)}), it holds
\begin{equation*}
	\left. \nabla_{\dot c(s)} \dot \gamma_s(l) \right|_{s \to 0}=
	\nabla_{\dot\gamma(l)} Y(l)  = \dot\gamma(l)\langle v, \dot\gamma(l)  \rangle/l +
	 (\coth l) (v- \dot \gamma(l)\langle v, \dot\gamma(l)\rangle).\label{eqn:nabla_dot_cs}
\end{equation*}
By substituting it, we obtain (\ref{eqn:proof_nabla_vHd^2(u,u)}).
\end{proof}


\paragraph{Proof of Theorem~\ref{thm:SC_hyperbolic}.}
We are going to bound 
\[
\sigma_x(u,v) := \frac{|\nabla_v Hf(u,u)|}{2 Hf (v,v)^{1/2} Hf(u,u)} \quad (u,v \in T_x).
\]
From $d(p,x) = l$, it holds $\|\dot \gamma(l)\| = 1$. 
We can also assume that $\|u\| = \|v\| = 1$.
Therefore, $u,v, \dot\gamma(l)$ can be assumed unit vectors in $\RR^3$ and  
represented in the spherical coordinate system as
$\dot \gamma(l) = (0, 0, 1)$, $u = (\sin \theta, 0,\cos \theta)$, 
$v = (\sin \varphi \cos \alpha, \sin \varphi \sin \alpha,\cos \varphi)$
for $\theta, \varphi \in [0, \pi]$ and $\alpha \in [0,2\pi]$.
By Lemma~\ref{lem:nablaHf_hyperbolic}, we have
\begin{eqnarray}
	Hf(v,v) & =&  \cos^2 \varphi + (l \coth l)  \sin^2 \varphi,\label{eqn:proof_Hd^2(v,v)}\\ 
	Hf(u,u) &= & \cos^2 \theta +   (l \coth l) \sin^2 \theta, \label{eqn:proof_Hd^2(u,u)}\\
  \nabla_{v} Hf(u,u) &=& \Phi(l) 
	\cos \varphi \sin^2 \theta + 2 \left( l - \Phi(l) \right)
	 \cos \theta \sin \varphi \sin \theta (- \cos \alpha ). \label{eqn:proof_nabla_vHd^2(u,u)}
\end{eqnarray}
%
By elementary calculus\footnote{Use $\sinh z = z + (1/3!)z^3+\cdots$ and $\cosh z = z^2 + (1/4!)z^4+\cdots$ to show $\lim_{z \to 0} z \coth z = 1$  and 
	$\Phi(0) = 0 \leq \Phi'(z)$ for $z \geq 0$}, 
we see
\begin{equation}
	0 \leq \Phi(z) \leq \min (1,z) \quad (z \in [0,\infty)).
\end{equation}
In (\ref{eqn:proof_nabla_vHd^2(u,u)}), the quantities
	$\Phi(l)$, $l - \Phi(l)$, $\sin \theta$, and $\sin \varphi$ are all nonnegative.   
	Thus
	\begin{equation}
		|\nabla_{v} Hf(u,u)| \leq  \Phi(l) 
	|\cos \varphi| \sin^2 \theta  
	+ 2 \left( l - \Phi(l) \right)
	\sin \varphi \sin \theta |\cos \theta|. 
	\end{equation}
	For $C := l \coth l \geq 1$, observe
\begin{eqnarray*}
	&&  \max_{\phi \in [0,\pi]} \frac{|\cos \phi|}
	{\sqrt{\cos^2 \phi + C \sin^2 \phi}} = 1,\quad   
	\max_{\phi \in [0,\pi]} \frac{\sin \phi}
	{\sqrt{\cos^2 \phi + C \sin^2 \phi}} = \frac{1}{\sqrt{C}}, \\
	&& \max_{\theta \in [0,\pi]} \frac{\sin \theta |\cos \theta|}
	{\cos^2 \theta + C \sin^2 \theta}  =
	\max_{\theta \in [0,\pi]}\frac{|\tan \theta|}{1+ C \tan^2 \theta}
	=  \max_{z \in [0,\infty)} \frac{z}{1+Cz^2}=
	\frac{1}{2\sqrt{C}}.
\end{eqnarray*}
Then 
	\begin{eqnarray*}
	\sigma_x(u,v) & \leq & \max_{\varphi, \theta \in [0,\pi]} \frac{\Phi(l) |\cos \varphi| \sin^2 \theta + 2(l - \Phi(l))\sin \varphi \sin \theta |\cos \theta|}{
	2 \sqrt{\cos^{2} \varphi + C \sin^2\varphi} \left(\cos^2 \theta + C \sin^2 \theta \right)} \\ 
	&\leq & \Phi(l)/2C + ( l - \Phi(l)) /2C =  (\tanh l)/2 \leq 1/2. 
	\end{eqnarray*}
Thus $f$ is $1/2$-SC.

Choose $\varphi=\pi/2$, $\tan^2 \theta = 1/C$, and $\alpha\in \{0,\pi\}$. From (\ref{eqn:proof_nabla_vHd^2(u,u)}) we have
\[
\sigma_x(u,v) = \frac{(l - \Phi(l)) |\cos \theta| \sin \theta}
{\sqrt{C} (\cos^2 \theta + C \sin^2 \theta)} =
\frac{l - \Phi(l)}{2C} = (l - \Phi(l))\frac{\tanh l}{2l}.
\]
For $l \to \infty$, it holds $\sigma_x(u,v) \to 1/2$. Hence the constant $1/2$ is tight.
This completes the proof of (1).

Showing (2), we consider the case of $u = v$; 
then $\varphi = \theta$ and $\alpha = 0$.
From (\ref{eqn:proof_nabla_vHd^2(u,u)}), we have 
\begin{equation}\label{eqn:proof_nabla_u_Hd^2(u,u)}
	\nabla_u Hf(u,u) = ( - 2 l+ 3 \Phi(l) )
	\cos \theta \sin^2 \theta.
\end{equation}
Then we have 
\begin{eqnarray*}
&&\sup_{u \in T_x} \sigma_x(u,u) =  \max_{\theta \in [0,\pi/2]} \frac{ | 2l - 3 \Phi(l) |  
 \tan^2 \theta}{2(1 + C \tan^2 \theta )^{3/2}} = \max_{z \in [0,\infty)} \frac{ 
	 | 2l - 3 \Phi(l) | z}{2(1 + C z )^{3/2}} \\
&& =  \frac{| 2l - 3 (\coth l + l - l\coth^2 l) |}{\sqrt{27} C}  = \frac{ | - 3/l - \tanh l  + 3 \coth l |}{\sqrt{27}},
\end{eqnarray*}
where the maximum of $z/(1+ C z)^{3/2}$  is attained at 
$z = 2/C = 2 (\tanh l)/l$.
The supremum of the last quantity is attained at $l  \to \infty$, and equals $2/\sqrt{27}$.
This means that $f$ is $\sqrt{4/27}$-WSC, tightly.

	By recovering the normalization $f \to (1/\kappa) f$, 
	we obtain tight $\sqrt{\kappa}/2$-SC and $\sqrt{4\kappa/27}$-WSC properties of the original $f$. 

\paragraph{Proof of Proposition~\ref{prop:tightness}.}
Let $x \in B_{p,R}^{\circ}$.
We bound $\sigma_{x}(u,u)$ from below. 
We use the notation $l, \theta$ in the previous proof. 
Let $\omega = \omega_{p,R,x}$ be defined by
$
\omega := R^2/2 - d(p,x)^2/2 = R^2/2 - f(x).
$
By (\ref{eqn:DF_HF}), Lemma~\ref{lem:Dd^2Hd^2} (2),  and  (\ref{eqn:proof_Hd^2(u,u)}), we have
\begin{eqnarray*}
	DF(u) &=& \frac{l \cos \theta}{\omega}, \\
	HF(u,u) &=&   \frac{1}{\omega} \left( \cos^2 \theta + l \coth l \sin^2 \theta \right) + \frac{l^2}{\omega^2} \cos^2 \theta\\
	&=& \frac{\omega+ l^2}{\omega^2}  \left( \cos^2 \theta + \left(\frac{\omega}{\omega+l^2} \right) l\coth l \sin^2 \theta \right).
\end{eqnarray*}
With  (\ref{eqn:nablaHF}) for $u=v$ and (\ref{eqn:proof_nabla_u_Hd^2(u,u)}), we have
\[
\nabla_u HF(u,u)= \frac{- 2 l+ 3 \Phi(l) }{\omega} 
\cos \theta \sin^2 \theta  + \frac{3 l }{\omega^2} \cos \theta ( \cos^2 \theta + l \coth l \sin^2 \theta  ) + \frac{2l^3}{\omega^3}\cos^3 \theta. 
\]
Take $\theta$ so that $\tan^2 \theta = \left(\frac{\omega+l^2}{\omega} \right)\frac{\tanh l}{l}$ (say).
Then $HF(u,u) = \frac{2(\omega+ l^2)}{\omega^2} \cos^2 \theta$, and
$\sigma_x(u,u)$ is given by 
\begin{eqnarray*}
	&& \frac{\omega^3}{2^{3/2}(\omega + l^2)^{3/2}}\left| 
	 \frac{-2l + 3 \Phi(l)}{\omega}
	\left(\frac{\omega+l^2}{\omega} \right)\frac{\tanh l}{l} \right. \left. +
	\frac{3l}{\omega^2} \left(\frac{\omega+l^2}{\omega} \right) + 
	 \left( \frac{3l}{\omega^2}+ 
	  \frac{2l^3}{\omega^3}\right)  \right| \\
	  && \geq \frac{ \omega}{2^{3/2}\sqrt{\omega+ l^2}}
	  \left| 2 \tanh l  - 3 \Phi(l) \frac{\tanh l}{l} \right| - \frac{3l} {2^{3/2}\sqrt{\omega + l^2}} 
	  - \frac{3\omega l + 2 l^3}{2^{3/2} (\omega + l^2)^{3/2}}\\
	  && \geq \frac{R^2 -l^2}{4\sqrt{R^2+l^2}}
	   \left| 2 \tanh l - 3 \Phi( l) \frac{\tanh l}{l}  \right| - \frac{5}{2^{3/2}}, 
\end{eqnarray*}
where $l/\sqrt{\omega+l^2} \leq 1$ and 
$(3\omega l + 2l^3)/(\omega+l^2)^{3/2}  =  2 (1+ (3/2)\alpha)/(1+ \alpha)^{3/2} \leq  2$
for $\alpha = \omega/l^2$ (see \cite[p. 177]{Nesterov04}).
Take $x$ such that $d(p,x)^2  = l^2 = R^2/2$ (say). For $R \to \infty$, 
the absolute value $|\cdot|$ converges to $2$, and
hence it holds $\sigma_x(u,u) = \Omega(R)$.

\section*{Acknowledgments}
The author was supported by JST PRESTO Grant Number JPMJPR192A, Japan.


\begin{thebibliography}{8}
\small
\bibitem{AMS07}
P. A. Absil, R. Mahony, R. Sepulchre:
{\em Optimization Algorithms on Matrix Manifolds.}
Princeton University Press, Princeton, NJ, 2008.

\bibitem{AGLOW}
Z. Allen-Zhu, A. Garg, Y. Li, R. Oliveira:
and A. Wigderson, 
Operator scaling via geodesically convex optimization, invariant theory and polynomial identity testing. {\tt arXiv:1804.01076}, 2018. (the conference version in STOC 2018)


\bibitem{AN00}
S. Amari and K. Nagaoka:
{\it Methods of Information Geometry.} 
American Mathematical Society, Providence, RI, 2000.

\bibitem{Boumal23}
N. Boumal: {\em An Introduction to Optimization on Smooth Manifolds.}
Cambridge University Press, Cambridge, 2023 (to appear).

\bibitem{BFGOWW}
P. B\"urgisser, C. Franks, A. Garg, R. Oliveira, 
M. Walter, and A. Wigderson:
Towards a theory of non-commutative optimization: 
geodesic first and second order methods 
for moment maps and polytopes. {\tt arXiv:1910.12375}, 2019. 
(the conference version in FOCS 2019)

\bibitem{BLNW}
P. B\"urgisser, Y. Li, H. Nieuwboer, and M. Walter:
Interior-point methods for unconstrained geometric programming and scaling problems.
{\tt arXiv:2008.12110}, 2020.


\bibitem{FR2021}
C. Franks and P. Reichenbach: 
Barriers for recent methods in geodesic optimization.
{\tt arXiv:2102.06652}, 2021.
(the conference version in CCC2021)


\bibitem{GGOW}
A. Garg, L. Gurvits, R. Oliveira, and A. Wigderson,
Operator scaling: theory and applications. 
{\em Foundations of Computational Mathematics}{\bf 20} (2020),223--290

\bibitem{HH22}
H. Hirai: Convex analysis on Hadamard spaces and scaling problems.
{\tt arXiv:2203.03193}, 2022.


\bibitem{JMM08}
H. Ji, J. H. Manton, and J. B. Moore: 
A globally convergent conjugate gradient method for 
minimizing self-concordant functions on Riemannian manifolds.
{\em IFAC Proceedings Volumes} {\bf 41} (2008), 14313--14318.

\bibitem{JMJ07}
 D. Jiang, J. B. Moore,  H. Ji: 
 Self-concordant functions for optimization on smooth manifolds. 
 {\em Journal of Global Optimization}  {\bf 38} (2007), 437--457. 


\bibitem{KOT13}
S. Kakihara, A. Ohara, T. Tsuchiya: 
Information geometry and interior-point algorithms in semidefinite programs and symmetric cone programs. 
{\em Journal of Optimization Theory and Applications} {\bf 157} (2013), 749--780.

\bibitem{KMY03}
P. Kumar, J. S. B. Mitchel, and E. A. Yildirim:
Approximate minimum enclosing balls in high dimensions using core-sets. 
{\em ACM Journal of Experimental Algorithmics} {\bf 8} 2003, 1.1-es.



\bibitem{Lee18}
J. M. Lee: {\em Introduction to Riemannian Manifolds.}  Springer, Cham, 2018. 

\bibitem{NH15}
F. Nielsen and G. Hadjeres:
Approximating covering and minimum enclosing balls in hyperbolic geometry.
{\em Geometric Science of Information},   
Lecture Notes in Computer Science {\bf 9389} (2015), 586--594. 

\bibitem{Nesterov04} 
Y. Nesterov: 
{\it Introductory Lectures on Convex Optimization. A Basic Course.}
Kluwer Academic Publishers, Boston, MA, 2004. 

\bibitem{Nesterov18} 
Y. Nesterov:
{\it Lectures on Convex Optimization.} 
Springer, Cham, 2018. 

\bibitem{NN94} 
Y. Nesterov and A. Nemirovskii:
{\it  Interior-Point Polynomial Algorithms in Convex Programming.} 
SIAM, Philadelphia, PA, 1994.

\bibitem{NT02} 
Y. Nesterov and M. J. Todd: 
On the Riemannian geometry defined by self-concordant barriers and interior-point methods. 
{\it Foundations of Computational Mathematics} {\bf 2} (2002), 333--361. 

\bibitem{NomizuSasaki}
K. Nomizu and T. Sasaki: {\em Affine Differential Geometry. Geometry of Affine Immersions.} Cambridge University Press, Cambridge, 1994. 

\bibitem{OT07}
A. Ohara and A. Tsuchiya: 
An information geometric approach to polynomial-time interior-point algorithms: 
complexity bound via curvature Integral. 
Research Memorandum No.1055, The Institute of Statistical Mathematics, 2007.

%

\bibitem{Renegar01}
J. Renegar: 
{\em A Mathematical View of Interior-Point Methods in Convex Optimization.} 
SIAM, Philadelphia, PA, 2001.

\bibitem{Rusciano19}
A. Rusciano: A Riemannian corollary of Helly's theorem, 
{\em Journal of Convex Analysis} {\bf 27} (2020), 1261--1275.

\bibitem{Sakai96} 
T. Sakai: 
{\it Riemannian Geometry.} 
American Mathematical Society, Providence, RI, 1996.


\bibitem{Sato21}
H. Sato: {\em Riemannian Optimization and Its Applications}, 
Springer, Cham, 2021.


\bibitem{Shima07} 
H. Shima:
{\it The Geometry of Hessian Structures.}
World Scientific, Hackensack, NJ, 2007.

\bibitem{SBH22}
M. Silva Louzeiro, R. Bergmann, and R. Herzog:
Fenchel duality and a separation theorem on Hadamard manifolds, 
{\it SIAM Journal on Optimization} {\bf 32} (2022), 854--873.


\bibitem{Udriste97}
C. Udriste: Optimization methods on Riemannian manifolds. 
{\em Algebras, Groups and Geometries} {\bf 14} (1997), 339--358.

\end{thebibliography}
\end{document}